\documentclass{article} 
\usepackage{iclr2027_conference,times}

\usepackage{amsmath,amsfonts,bm}

\def\eqref#1{equation~\ref{#1}}

\def\1{\bm{1}}

\DeclareMathAlphabet{\mathsfit}{\encodingdefault}{\sfdefault}{m}{sl}
\SetMathAlphabet{\mathsfit}{bold}{\encodingdefault}{\sfdefault}{bx}{n}

\usepackage{hyperref}
\usepackage{url}

\title{Stability of the Monge Map in Semi-Dual Optimal Transport}

\author{Anton Selitskiy 
\\
Dept. of Electrical and Computer Engineering\\
Stony Brook University\\
Stony Brook, NY 11790, USA \\
\texttt{anton.selitskiy@stonybrook.edu} \\
\And
David Millard\\
Dept. of Electrical and Computer Engineering \\
University of Rochester \\
Rochester, NY 14627, USA \\
\texttt{dmillard@ur.rochester.edu} \\
}

\usepackage{hyperref}       
\usepackage{url}            
\usepackage{booktabs}       
\usepackage{nicefrac}       
\usepackage{microtype}      
\usepackage{xcolor}         

\usepackage{graphicx}
\usepackage{amsmath}
\usepackage{amsthm}
\usepackage{amssymb}
\usepackage{subcaption}

\newtheorem{theorem}{Theorem}

\newtheorem{remark}{Remark}

\iclrfinalcopy 
\begin{document}

\maketitle

\begin{abstract}
This paper shows that the semi-dual formulation of the optimal transport problem has a degenerate saddle-point structure, and that its numerical solution is equivalent to solving a constrained optimization problem. We derive necessary and sufficient conditions for the convergence of Monge maps without requiring optimality of the dual potential. This analysis helps explain why, in practice, numerical algorithms often require more iterations to update the transport map than the potential.
\end{abstract}

\section{Introduction}

Recently, optimal transport (OT) has become a powerful tool in machine learning, statistics, and signal processing. From a theoretical perspective, many modern generative models --- including variational autoencoders (VAEs), generative adversarial networks (GANs), diffusion models, and flow-based models --- can be interpreted as approximations of transport processes between probability distributions. For instance, GANs can be viewed as learning a map that transports a simple latent distribution (e.g., uniform or Gaussian) to complex data distribution (e.g., a highly nonlinear data generating distribution). Similarly diffusion models describe stochastic flows that gradually transform one distribution into another. These connections suggest that \emph{optimal transport provides a unifying mathematical framework for generative modeling}, where learning corresponds to approximating transport maps or transport dynamics between distributions.

This work is motivated by neural optimal transport (NOT) algorithm introduced in \citet{rout2021generative}, which is closely related to the $(2,2)$-WGAN from \cite{mallasto2019qpwgan} and other WGAN formulations. Later, NOT algorithm was extended to weak formulations in \citet{korotin2023neural}. These approaches rely on the semi-dual formulation of optimal transport\footnote{This formulation is already present in the seminal work of \citet{Kan04Transl}; the optimization itself is stated explicitly in \citet{refId0}, while the term   “semi-dual” is introduced in \citet{cuturi2016semidual}.} together with amortized optimization techniques introduced in \citet{threeplayergan}, later interpreted as measurable selection mechanisms in the theory of integrands. 

Altogether, these works exploit the existence (and, for quadratic cost, uniqueness) of optimal transport maps, but do not provide convergence guarantees for the learned maps. Many  papers containing convergence estimates implicitly assume that the transport map exists and that the optimization problem is well-posed; however, these assumptions are mathematically nontrivial and may fail in practice, leading to misleading or impractical conclusions (see Sec.~\ref{sec:discussion}).

In this paper, we show that  the semi-dual optimal transport problem possesses a degenerate saddle-point structure: at the optimal transport map, the objective functional becomes independent of the potential.  We provide an estimate  of the transport map which is independent of the potential:
\begin{equation}\label{eq:intro}
\|T - T^\star\|_{L_2}^2 \lesssim  {\left|\mathbb{E}[c(x,T(x))] - W_2(\mu,\nu)^2\right|} + d_{KR}(T_{\sharp}\mu,\nu).
\end{equation}



While many works assess convergence through the approximation gap and the discrepancy between generated and target distributions, to the best of our knowledge such estimates have not been formulated explicitly in the form of \eqref{eq:intro}.

Early works on the numerical solution of the semi-dual OT problem approximated the conjugate potential using an inner optimization loop \cite{SemiGAN}. Later, with neural network parametrizations it was observed that performing multiple updates of the transport map leads to better performance \cite{makkuva2020icnn}. We provide a theoretical explanation for this phenomenon by interpreting the problem as a constrained optimization task and relating it to a continuous-time two-timescale dynamical system.

\section{Optimal Transport Problem}
\label{Sec:OT}

\subsection{Monge Formulation} 
We will assume that $X$ is a complete separable metric space. Such spaces are referred to as Polish spaces. In practice, we consider either $X=\mathbb{R}^n$ or a compact subset of $\mathbb{R}^n$.
By $\mathcal{P}(X)$ we denote the space of probability measures defined on the Borel $\sigma$-algebra $\mathcal{B}(X),$ i.e., the smallest $\sigma$-algebra generated by open subsets of $X.$ $\mathcal{P}_p(X)$ denotes subspace of $\mathcal{P}(X)$ having $p$th moments.

Assume that $Y$ is a Polish space and a lower semicontinuous (l.s.c.)\footnote{Continuous functions are l.s.c.} cost function is defined:
 $c\colon X\times Y\to [0,\infty).$   
Later, we assume that $X$ and $Y$ are compact sets in $\mathbb{R}^n$ endowed with metric $d(x,y)$ and cost
\begin{equation}\label{p_dist}
    c(x,y) = \frac{1}{p}d(x,y)^p. 
\end{equation}
Letting $\mu\in\mathcal{P}(X)$ and $\nu\in\mathcal{P}(Y)$ be probability measures we can consider two probability spaces $(X,\mathcal{X},\mu)$ and $(Y,\mathcal{Y},\nu)$ with $\mathcal{X}=\mathcal{B}(X)$ and $\mathcal{Y}=\mathcal{B}(Y).$
For a $\mathcal{X}|\mathcal{Y}$-measurable function $T\colon X\to Y,$ the measure $\nu$   such that $\nu(B) = \mu (T^{-1}(B))$ for all $B\in\mathcal{Y}$ is called a push-forward measure of $\mu$ by $T$.  In this case, the following notation is used: $\nu = T_{\sharp}\mu\ (= \mu\circ T^{-1})$. 

Denote by $T(\mu,\nu)$ the space of measurable (Borel) functions $T\colon X\to Y,$ for which $\nu$ is a push-forward measure of $\mu.$  Elements of $T(\mu,\nu)$ will be referred to as \emph{transport maps}.


    The Monge problem is to find a transport map $T^{*}\colon X\to Y$ that minimizes the transport cost 
\begin{equation}\label{monge_problem}
    MP(\mu,\nu,c) = \underset{T\in T(\mu,\nu)}{\inf} \int\limits_{X} c(x,T(x))d\mu(x).
\end{equation}
If the infimum is reached on some function $T^{*},$ then $T^{*}$ is called \textit{the optimal transport map.} 
Sometimes, the class $T(\mu,\nu)$ is empty (see Table~\ref{tab:existence}).

\subsection{Kantorovich Formulation} %
\label{sec:ch1_Kant_form}

Denote by $\Pi(\mu,\nu)$ the set of probability measures $\pi$ on the product space $(X\times Y,\mathcal{X}\otimes\mathcal{Y})$ with prescribed marginal distributions:
$ 
\pi(A\times Y)=\mu(A),$ and $ 
\pi(X\times B)=\nu(B),
\quad \forall A\in\mathcal{X},\ B\in\mathcal{Y}.
$ 
 Measures $\pi\in\Pi(\mu,\nu)$ are called \textit{transport (or transportation) plans}.\footnote{
The set $\Pi(\mu,\nu)$ is nonempty; for instance it contains the product measure $\mu\otimes\nu$.}


The Kantorovich problem consists in finding a transport plan $\pi^{*}$ that minimizes the total transport cost 
\begin{equation}\label{kant_fofm}
KP(\mu,\nu,c)
=
\inf_{\pi\in\Pi(\mu,\nu)}
\int_{X\times Y} c(x,y)\,d\pi(x,y).
\end{equation}
The integral can be written as 
$
\mathbb{E}[c(\xi,\eta)],
$
where $(\xi,\eta)\sim\pi$. A plan $\pi^{*}$ at which the minimum is attained is called an \textit{optimal transport plan}. Under  assumptions above\footnote{For unbounded domain additionally $
 c(x,y)\le a(x)+b(y)
 $ for some integrable functions $a\in L_1(\mu)$ and $b\in L_1(\nu).$} the infimum in \eqref{kant_fofm} is actually attained \cite[Theorem~5.10~(iii)]{villani:book}. 





\subsection{Dual Problem}
\label{Sec:dual}

Problem~\eqref{kant_fofm} can be reformulated as a maximization problem (for generality we write supremum, but it is attained under very general assumptions:
\begin{equation}\label{dp_general}
    DP(\mu,\nu,c) = \underset{\substack{\varphi\in C(X), \psi\in C(Y)\\\varphi(x)+\psi(y)\leq c(x,y)}}{\sup}\left(\int\limits_{X}\varphi d\mu + \int\limits_{Y}\psi d\nu\right)
\end{equation}
where $X$ and $Y$ are compact sets, in general case $\varphi$ and $\psi$ should be absolutely integrable.

Under our assumptions (see, e.g., \cite[Theorem~5.10~(iii)]{villani:book}),
$ 
    DP(\mu,\nu,c) = KP(\mu,\nu,c).
$ 

\subsection{Semi-Dual Problem}
For  functions $\varphi\colon X\to \mathbb{R}\cup\{\pm\infty\}$ and $\psi\colon Y\to \mathbb{R}\cup\{\pm\infty\}$ define $c$\textit{-transform}
\begin{equation}
    \varphi^c(y) = \underset{x\in X}{\inf} [c(x,y) - \varphi(x)],\qquad \psi^c(x) = \underset{y\in Y}{\inf} [c(x,y) - \psi(y)].
\end{equation}
    A function $\varphi(x)$ is said to be $c$-concave, if there exists a function $\psi(y),$ such that $\varphi(x) = \psi^c(x).$
For $X$ and $Y$ bounded subsets in $\mathbb{R}^n$ the dual problem of optimal transport can be formulated as semi-dual (see):\footnote{For unbounded domains, we should assume $\psi \in L_1(d\nu).$} 
\begin{equation}
\label{eq:dual}
SDP(\mu,\nu,c) = \sup_{\psi \in C(Y)} 
\left( \int_X \psi^c(x)\, d\mu + \int_Y \psi(y)\, d\nu \right).
\end{equation}

\subsection{Connection between the Monge and Kantorovich Problems}

The following theorem establishes the relation between the Monge and Kantorovich formulations of optimal transport.

\begin{theorem}[\cite{Pratelli2007}]
Let $X$ and $Y$ be Polish spaces and let $\mu\in\mathcal P(X)$ and $\nu\in\mathcal P(Y)$ be probability measures. 
Assume that $\mu$ is nonatomic and that the cost function 
$c:X\times Y\to[0,\infty)$ is continuous. Then  $MP(\mu,\nu,c) = KP(\mu,\nu,c).$
\end{theorem}

The theorem states that the infimum value of the Monge problem coincides with the minimum value of the Kantorovich problem. In particular, although the Monge problem may fail to admit a minimizing map, it is possible to approximate the optimal transport plan by transport maps  whenever the source measure is nonatomic.
In Table~\ref{tab:existence}, we summarize the theoretical results on the existence of the optimal transport map according to \cite{BogKol12}.

\begin{table}[h]
\centering
\setlength{\tabcolsep}{5pt}
\renewcommand{\arraystretch}{1.1}
\caption{Existence and equivalence properties in optimal transport under different assumptions.}
\label{tab:existence}
\begin{tabular}{l|c|c|c|c|c|c}
\toprule
$\mu,\nu$ & $c(x,y)$ & $\pi^\star$ & $\varphi^\star,\psi^\star$ & $T^\star$ & $KP=DP$ & $KP=MP$ \\
\midrule
general & continuous & $\exists$ & $\exists$ & sometimes & yes & $\leq $ \\
continuous & l.s.c. & $\exists$ & $\exists$ & sometimes & yes & $\leq $ \\
$\mu$ non-atomic & continuous & $\exists$ & $\exists$ & sometimes & yes & yes \\
$\mu$ no $(n\!-\!1)$ mass & $\tfrac{1}{p}|x-y|^p,\ p>1$ & $\exists!$ & $\exists$ & $\exists!$ & yes & yes \\
continuous & $|x-y|$ & $\exists$ & $\exists$ & $\exists$ & yes & yes \\
\bottomrule
\end{tabular}
\end{table}

When  $X$ and $Y$ are  subsets of $\mathbb{R}^n$ and cost is quadratic $c(x,y) = \tfrac{1}{2}\|x-y\|^2$, the Kantorovich problem admits additional structure.\footnote{\cite{Bernier91} for compact sets and \cite{McCann95} in case of $X=Y=\mathbb{R}^n.$ } 

\begin{theorem}[Brenier--McCann]
\label{thm:monge}
Let $X$ and $Y$ be compact sets in $\mathbb{R}^n$ or $X=Y=\mathbb{R}^n,$ $c(x,y) = \tfrac{1}{2}\|x-y\|^2,$ and suppose $\mu$ vanishes on all sets of Hausdorff dimension at most $n-1$ (for example, if $\mu$ is absolutely continuous with respect to Lebesgue measure). Then there exists a Borel map $T\colon X\to Y$ such that $\nu = T_{\sharp}\mu.$ Moreover:
    (i) ~$T = \nabla u$ for some convex function $u$; 
    (ii) $T$ is $\mu$-a.e. unique;
    (iii) $T$ is the unique solution to the Monge problem.

If $\nu$ satisfies the same condition, then there exists an inverse map $S=\nabla v$, where $v$ is the Legendre--Fenchel dual of $u$, such that $S\circ T(x)=x$ and $T\circ S(y)=y$ almost surely.
\end{theorem}

This theorem  provides insights into situations in which the Monge map may fail to exist.
\cite[Theorem~5.10]{villani:book} says that the support\footnote{In case of a separable metric space $X$, \textit{support of a measure} $\mu\in\mathcal{P}(X)$ can be defined as the smallest closed set $B\subset X$ with $\mu(B)=1.$} of the optimal plan belongs to a $c$-superdifferential of the potential:
\begin{equation}
        \operatorname{sp}\mu\subset\partial^c\varphi = \left\{ (x,y)\in X\times Y: \varphi(x) + \varphi^c(y)= c(x,y)\right\}.
    \end{equation}
In case of quadratic cost, a $c$-concave function can be represented as $\varphi(x)=\frac{|x|^2}{2}-u(x),$ where $u$ is convex function. Because a convex function is Lipschitz, the potential $\varphi$ is Lipschitz, that means it is differentiable almost everywhere, and the $c$-superdifferential coincides with $\nabla\varphi$ at almost every point. Consequently, the support of the optimal plan $\operatorname*{sp}\pi$ is concentrated on the graph of $\nabla\varphi,$ up to a set of measure zero where the superdifferential may be multivalued. In this case, the set of points where the transport map is non-unique has zero measure. 

In contrast, if the set of points where $\varphi$ is not differentiable has positive measure, then the $c$-superdifferential is genuinely multivalued on a set of positive measure. As a result, the support of $\pi$ cannot be represented as the graph of a function, and therefore no Monge map exists.

\section{Towards a Minimax Formulation}
\label{Sec:continuous}

Given that many applications with bounded domains, e.g., in images the components belong to $[0,1]$ or $[0,255],$ from now we assume that $X$ and $Y$ are compact subsets in $\mathbb{R}^n.$ 

\subsection{Saddle-Point Formulation}

By interchange theorem (see \citet[Th. 14.60]{Rockafellar2009} or \citet{Rockafellar1976integral}), one can rewrite (\ref{eq:dual}) in a form that involves measurable selections: 
\begin{equation}
\label{sup_inf}
SDP(\mu,\nu,) = \sup_{\psi \in C(Y)}  
\left( \inf_{t \in \mathcal{M}(X,Y)}\int\limits_X \big[ c(x,t(x)) - \psi(t(x)) \big] d\mu + \int\limits_Y \psi(y)\, d\nu \right).
\end{equation}
Where $\mathcal{M}(X,Y)$ denotes all measurable functions from $X$ to $Y$. It can be replaced by $L_p$ if the integral converges ($p\geq 1$).

In the quadratic case $c(x,y) = \frac{1}{2}\|x-y\|^2$, one expects (and can show under the stated assumptions) that the infimum  in (\ref{sup_inf}) is attained at the optimal Monge map $t(x)=T^\star(x)$, while the supremum is attained at the Kantorovich potential $\psi^\star$. The following theorem formalizes it.

\begin{theorem}
\label{thm:not_clean}
Let $X$ and $Y$ be compact subsets in $\mathbb{R}^n$, and let $\mu,\nu$ be probability measures  both absolutely continuous with respect to the Lebesgue measure. Let the cost $c\colon X\times Y\to \mathbb{R}$ be l.s.c.
For measurable maps $t\colon X\to Y$ and $\psi\colon Y\to \mathbb{R}\ (\cup\{\pm\infty\}),$ such that $\psi \in C(Y),$ define
\begin{equation}\label{eq:min_max_func}
F(\psi,t)
=
\int\limits_X c(x,t(x))\,d\mu
+
\int\limits_Y \psi(y)\,d(\nu - t_\sharp \mu).
\end{equation}
Then:
\begin{enumerate}
    \item[(i)] The minimax value satisfies
\begin{equation}\label{eq:th15}
\sup_{\psi} \inf_{t} F(\psi,t)
=
\sup_{\psi}
\left(
\int_X \psi^c \, d\mu + \int_Y \psi \, d\nu
\right),
\end{equation}
which coincides with the Kantorovich dual problem.

\item[(ii)] For all $\psi$
$    \underset{t}{\inf} F(\psi,t)\leq \underset{t}{\inf} F(\psi^{cc},t).
$ 
with equality reached on $c$-concave functions $\psi$ only.
\item[(iii)]  If $t_{\sharp}\mu\ne \nu,$
\begin{equation}\label{eq:separ}
    \sup_{\psi} F(\psi,t) = +\infty.
\end{equation}

\item[(iv)] If in addition the cost 
 $   c(x,y)=\frac{1}{p} |x-y|^p,\quad 1< p<+\infty.
$ 
There exists an optimal potential $\psi^\star$ and a measurable map $T^\star$ such that for all $\psi$ and $t$
\begin{equation}
F(\psi^\star,t)
\ge
F(\psi^\star,T^\star)
=
F(\psi,T^\star).
\end{equation}
 Moreover,
$    T^\star(x) = \arg\underset{y}{\min} [c(x,y) - \psi^\star(y)] \quad \mu\text{-a.e.}
$ 
and $T^\star$ coincides with the Brenier optimal transport map from $\mu$ to $\nu$
\begin{equation}\label{eq:map_via_grad}
    T^\star(x) = x-|\nabla\psi^{\star c}(x)|^{q-2}\nabla\psi^{\star c}(x) \quad \mu\text{-a.e.}, \quad \frac{1}{q} = 1-\frac{1}{p}.
\end{equation}

\item[(v)] The map $T^\star$ is unique $\mu$-almost everywhere and is the unique minimizer of $t\mapsto F(\psi^\star,t).$  
\end{enumerate}

\end{theorem}

The equality in \textit{(iv)} reveals a fundamental structural property of the formulation: flatness of the functional in the dual variable. More precisely, at the optimal transport map $t=T^\star,$ the value of the functional becomes independent of $\psi.$  This degeneracy closely parallels well-known instability phenomena in generative adversarial networks, where the discriminator may continue to change without affecting the objective once the generator approaches optimality. The works we survey either implicitly treat the saddle point as strict or use formulations that do not foreground the constrained-optimization structure (Section~\ref{sec:5.5}) making the degeneracy apparent; we are not aware of an explicit treatment elsewhere in the semi-dual optimal transport literature.

This observation has important consequences for existing convergence results. In particular, several works --- including \citet[Assumption~3]{fan2021neuralmonge_arxiv}, \citet[Theorem~4.3]{rout2021generative} (convexity
assumption), \citet[Theorem~4.6]{tarasov2026a}, and \citet[Lemma~1]{korotin2021tuning} --- rely on properties derived under the assumption that the potential $\psi$ is optimal. However, due to the flatness of the objective, optimality of $\psi$ is neither necessary nor enforced by the optimization dynamics once the transport map is close to optimal. As a result, convergence of the map can be achieved without convergence of the potential, as demonstrated in Section~\ref{sec:discussion}.

Connection (\ref{eq:map_via_grad}) is used in \citet{makkuva2020icnn}, \citet{korotin2019w2gn}, \citet{SemiGAN} and other works (see Table~\ref{tab:OT_GAN}). 

\section{Stability of transport maps}

\subsection{Convergence}

Recall the notion of the Kantorovich--Rubinstein distance in the space of probability measures $\mathcal{P}(X).$
\begin{equation}
    d_{KR}(\mu,\nu) = \sup_{\substack{\|\varphi\|\leq 1,\\ |\varphi(x)-\varphi(y)|\leq d(x,y)}} \left( \int\limits_{X}\varphi(x)\,d\mu(x) - \int\limits_{X}\varphi(x)\,d\nu(x)\right).
\end{equation}
Kantorovich introduced metric in $\mathcal{P}_1(X)$ as $W_1(\mu,\nu) = KP\bigl(\mu,\nu, d(x,y)\bigr)$ and using dual formulation (\ref{dp_general}) showed that
\begin{equation}
    W_1(\mu,\nu) = \sup_{ |\varphi(x)-\varphi(y)|\leq d(x,y)} \left( \int\limits_{X}\varphi(x)\,d\mu(x) - \int\limits_{Y}\varphi(y)\,d\nu(y)\right).
\end{equation}
It has been generalized to \textit{Kantorovich power metric}\footnote{The distance $W_1$ (denoted as $W$ in \cite{Kan04Transl}) 
was interpreted by Kantorovich as the minimal ``work'' required to transport mass. In the modern literature, the notation $W_p$  is often mistakenly referred to as the Wasserstein distance. This misattribution was mentioned in multiple publications including L.~Vasershtein's friend remark: ``It is especially ironic to find the Kantorovich metric called
the Vasershtein metric\ldots Vasershtein’s interesting article \cite{Vaserstein1969}\ldots does
contain in passing a definition of the Kantorovich metric\ldots  But there is no
definition of power metrics'' \cite{vershik2013long}.}  $W_p(\mu,\nu)^p = KP\bigl(\mu,\nu,\frac{1}{p}d(x,y)^p\bigr).$
Furthermore, $d_{KR}$ is the metric in $\mathcal{P}(X)$ and convergence in it is equivalent to  weak convergence. Similarly, $W_p$ is a metric in $\mathcal{P}_p(X)$ and convergence in it in addition to weak convergence guarantees the convergence of the first $p$ moments.

\begin{theorem}\label{th:conv} Let the assumptions of Theorem~\ref{thm:not_clean} hold, and $c(x,y)=d(x,y)^2/2.$   
In addition, assume that measures $\mu$ and $\nu$ have densities $p\in C^{0,\alpha}(X)$ and $q\in C^{0,\alpha}(Y)$  with $\alpha\in (0,1).$  Additionally, $p,\ q$ are bounded away from zero on $X,\ Y$, and $X,\ Y$ have $C^{2,\alpha}$ smooth boundaries and are uniformly convex. Assume that $t_k\in L^2(\mu)$ and the family $\{\psi_k\}$ is uniformly $L$-Lipschitz.\footnote{In practice, parametrization $t_{\theta_k}, \psi_{\phi_k}$ by neural networks with parameters $\theta_k$ and $\phi_k$ is used.} Then the following estimate is true
\begin{equation}
\|t_{k} - T^\star\|_{L^2(\mu)}^2 \leq C \left( \left|F(t_{k},\psi_{k}) - W_2(\mu,\nu)^2\right| + d_{KR}(t_{k\sharp}\mu,\nu)\right).
\end{equation}
\end{theorem}

In practice, it is rarely possible to compute the exact Kantorovich distance. Instead, we rely on observable quantities derived from the optimization objective. 

We first observe that the functional $F$ consists of two components with distinct roles. Next we notice that the first term acts as a stabilizing force, favoring maps close to the identity, while the second term enforces the marginal constraint $t_\sharp\mu=\nu$  preventing collapse to the trivial solution. Finally, since the constraint is only enforced approximately, there is no guarantee that the transport cost reaches the optimal value at $F(t,\psi) = W_2(\mu,\nu)^2.$ Therefore, convergence cannot be assessed using a single quantity. Simultaneous improvement in both quantities provides empirical evidence of convergence toward the optimal transport map.

\begin{remark}
Some non-smooth domains can be considered in Theorem~\ref{th:conv}. For example, multidimensional cubes $X=Y=[0,1]^d$ due to the mirroring argument. 
\end{remark}

\begin{remark}\label{remk2}
    The reverse estimate is true under the same assumptions, i.e., from $t_{k}\to T^\star$ in $L^2(\mu)$ and $\psi_{k}$ is uniformly $L$-Lipschitz it follows that $F(t_{k},\psi_{k}) \to W_2(\mu,\nu)^2$ and $d_{KR}(t_{k\sharp}\mu,\nu)\to 0.$
\end{remark}



\subsection{Hyperparameters of the Numerical Algorithm}

In the NOT algorithm by \citet{korotin2021neuralot}, several gradient descent steps are performed to update $t$ before a single gradient ascent step is applied to $\psi.$ Statement~(iii) of
Theorem~\ref{thm:not_clean} provides insight into this design choice. 

During training, the constraint $t_\sharp\mu=\nu$ is not satisfied so the maximization problem is unbounded. Assume that $d(t_\sharp\mu)=\rho_t(y)dy$ and $d\nu=q(y)dy.$ If we look at the gradient,
$
    \nabla_{\psi}F = \rho_t - q
$, 
we see that the gradient ascent in $\psi$ causes an increase in that direction without any restoring force, due to absence of concavity of $F$ in $\psi.$ 
This can cause instability if the  variable $t$ is not sufficiently close to optimal value. On the other hand, performing of multiple steps of optimization in $t$ reduces both the transport cost and the constraint violation  $\|\rho_t-q\|,$ which results in smaller step in $\psi.$

This naturally leads to a two-timescale interpretation of the training dynamics: the primal variable $t$ evolves on a faster timescale, stabilizing the system, while the dual variable $\psi$ evolves more slowly to enforce the constraint (see Sec.~\ref{sec:discussion} for dynamical system analogy).

One step of update looks as follows:
\begin{equation}
    \begin{split}
        &t^{new}_0 = t^{old},\ \
        t^{new}_{k+1} = t^{new}_{k} - \eta_{t_{k+1}}\nabla_t F(t^{new}_k, \psi^{old}),\ k=0,\ldots,K-1,\ \
        t^{new} = t^{new}_{K-1},\\
        &\psi^{new} = \psi^{old} + \eta_{\psi}\nabla F(t^{new},\psi^{old}).
    \end{split}
\end{equation}
How to choose  $K$ and $\eta_t/\eta_\psi$? We discuss this in the next section with further experimentation in Sec.~\ref{sec:Experiments}. Also note, if $\eta_\psi \ll K\eta_t,$ the system behaves as
$
    \min_{t} F(t,\psi)
$
with slow varying $\psi.$ The value $\kappa = K\eta_t/\eta_\psi$ characterizes the relative time scales of the transport map and potential updates.

\section{Experimental Results and Discussion}
\label{sec:discussion}

\subsection{Different formulations of NOT and connection with GANs}

We classify static 
neural optimal transport (NOT) solvers into the following groups:
\begin{enumerate}
    \item[I.] min--max optimization over the transport map $T$ and the potential $\varphi$ (or $\psi$);
    \item[II.] min--max optimization over $\varphi$ (or $\psi$) with the transport map defined as $T=\nabla\varphi$ (or $S=\nabla\psi$) \cite{makkuva2020icnn};
    \item[III.] regularized (and possibly unbalanced) formulations. The effect of regularization and the bias it introduces require separate investigation.
\end{enumerate}

In the case of quadratic cost, two equivalent formulations are commonly used:
\begin{enumerate}
    \item Minimization of the squared distance;
    \item Maximization of the dot product, also known as maximum crossentropy functional (MCF).
\end{enumerate}
In the latter case, optimization is performed over convex functions \cite{tarasov2026a}. All the results derived in this paper remain valid under this formulation, except for statement (iii) of Theorem~\ref{thm:not_clean} (see Sec.~\ref{sec:app_mcf}).

The connection between NOT and generative adversarial networks (GANs) can be understood through the existence of measurable transport maps. In particular, classical results originating from \cite{vonneumann1938}  imply that there exists a measurable function $G$ mapping a simple distribution $\mu$ (e.g., uniform $U[0,1]$) to an arbitrary non-atomic distribution $\nu$. In practice, classical GANs approximate such a map using neural networks and train it by minimizing the Jensen--Shannon divergence\footnote{Up to a constant, because of $$JS(\nu||G_\sharp\mu)=\sup_{D\colon X\to [0,1]}\left( \int_X \log D(x)d\nu + \int_X \log(1-D(x)) d(G_\sharp\mu) \right) + \log 2.$$} between $G_{\sharp}\mu$ and $\nu:$
\begin{equation}
    \min_{G\in\mathcal{M}(X,Y)}\sup_{D\colon Y\to (0,1)} 
    \left(
    \mathbb{E}_{y\sim\nu}[\log D(y)] 
    + 
    \mathbb{E}_{x\sim\mu}[\log(1-D(G(x)))]
    \right).
\end{equation}
WGAN replaces the divergence with the Kantorovich distance $W_1(\nu, G_\sharp\mu)$:
\begin{equation}
    \min_{G\in\mathcal{M}(X,Y)}\sup_{\varphi\in \operatorname*{Lip}_1(Y)} 
    \left(
    \mathbb{E}_{y\sim\nu}[\varphi(y)] 
    - 
    \mathbb{E}_{x\sim\mu}[\varphi(G(x))]
    \right).
\end{equation}
If $X,Y\subset\mathbb{R}^n$ have the same dimensionality, the optimal transport formulation can be used:
\begin{equation}
W_1(\mu,\nu) = \sup_{\psi\in C(Y)} \inf_{t}
\left( 
\int_X |x-t(x)|\, d\mu 
- \int_X \psi(t(x))\,d\mu
+ \int_Y \psi(y)\,d\nu
\right).
\end{equation}
In this setting, minimizing the WGAN objective enforces the marginal condition  $G_\sharp\mu=\nu.$ The optimal transport map $T^\star,$ when it exists, is one particular generator satisfying this condition. More generally, $(q,p)$-WGAN \cite{mallasto2019qpwgan} corresponds to the optimal transport problem with cost
$     c(x,y) = \frac{1}{p}d_q(x,y)^p,$ where $  
    d_q(x,y)=\left(\sum_{i=1}^n |x_i-y_i|^q\right)^{1/q}.
$

We summarize the training objectives for comparison in Table~\ref{tab:OT_GAN}. All these formulations share a common structural feature: they define degenerate saddle-point problems, where the discriminator (or potential) is not uniquely determined when the generator (or transport map) is close to optimal.

\begin{table}[t]
\setlength{\tabcolsep}{5pt}
\renewcommand{\arraystretch}{1.1}
\centering
\setlength{\tabcolsep}{4pt}
\caption{NOT solvers and related models}
\label{tab:OT_GAN}
\begin{tabular}{l|c|c}
\toprule
Group & Model & Objective \\
\midrule

I & WGAN &
$\displaystyle 
\inf_{G}\ \sup_{\varphi\in \mathrm{Lip}_1(Y)}
\Big(
\mathbb{E}_{y\sim\nu}[\varphi(y)]
-
\mathbb{E}_{x\sim\mu}[\varphi(G(x))]
\Big)$ \\

I & $(p,q)$-WGAN &
$\displaystyle 
\sup_{\psi\in C(Y)}\ \inf_{G}
\Big(
\mathbb{E}_{x\sim\mu}\big[\tfrac{1}{p} d_q(x,G(x))^p\big]
+
\mathbb{E}_{y\sim\nu}[\psi(y)]
-
\mathbb{E}_{x\sim\mu}[\psi(G(x))]
\Big)$ \\

I & NOT &
$\displaystyle 
 \sup_{\psi\in C(Y)}\ \inf_{t}
\Big(
\mathbb{E}_{x\sim\mu}\big[c(x,t(x))\big]
+
\mathbb{E}_{y\sim\nu}[\psi(y)]
-
\mathbb{E}_{x\sim\mu}[\psi(t(x))]
\Big)$ \\

I & MCF &
$\displaystyle 
 \inf_{\substack{v\in C(Y)\\ v\ \mathrm{convex}}}\ \sup_{t}
\Big(
\mathbb{E}_{x\sim\mu}[\langle x,t(x)\rangle]
+
\mathbb{E}_{y\sim\nu}[v(y)]
-
\mathbb{E}_{x\sim\mu}[v(t(x))]
\Big)$ \\

I & MCF  &
$\displaystyle 
 \inf_{\substack{u\in C(X)\\ u\ \mathrm{convex}}}\ \sup_{s}
\Big(
\mathbb{E}_{y\sim\nu}[\langle s(y),y\rangle]
+
\mathbb{E}_{x\sim\mu}[u(x)]
-
\mathbb{E}_{y\sim\nu}[u(s(y))]
\Big)$ \\

II & MCF  &
$\displaystyle 
\inf_{\substack{u\in C(X)\\ u\ \mathrm{convex}}}
\ \sup_{\substack{v\in C(Y)\\ v\ \mathrm{convex}}}
\Big(
\mathbb{E}_{y\sim\nu}[\langle \nabla v(y),y\rangle]
+
\mathbb{E}_{x\sim\mu}[u(x)]
-
\mathbb{E}_{y\sim\nu}[u(\nabla v(y))]
\Big)$ \\

\bottomrule
\end{tabular}
\end{table}

Finally, we note an important distinction: in GANs and WGANs, the generator $G$ typically maps a low-dimensional latent distribution to a high-dimensional data space. While this guarantees the existence of such a map, it does not ensure that it coincides with an optimal transport map. In contrast, $(q,p)$-WGAN and NOT formulations operate on spaces of equal dimensionality, where the metric  as a cost function is used.

\subsection{Connection with  Flow-Matching}
Flow-matching approximates OT trajectories with an ODE (\cite{FM}). NOT can be viewed as the solution of ODE in one step. The continuous limit of the descrete optimiszation scheme is the following dynamic system
\begin{equation}\label{eq:dynamic_syst}
    \begin{cases}
        \dfrac{\partial \rho_s}{\partial s} + \nabla(\rho_s v_s)=0, \text{ where } v_s(y)=-(y-t_s^{-1}(y)) + \nabla\psi_\tau(y), \\[10pt] 
        \dfrac{\partial \psi_\tau}{\partial \tau} = \kappa^{-1}(q(y)-\rho_s(y)).
    \end{cases}
\end{equation}
The importance of two-timescale updates has been widely observed in GAN training, where multiple discriminator updates per generator step are required for stability \cite{arjovsky2017wasserstein, gulrajani2017improved}. 
This behavior was formalized in the two time-scale update rule \cite{HeuselRUNKH17}, which shows that convergence depends on the relative step sizes of the two players. 
Similar empirical strategies are used in neural optimal transport \cite{korotin2021neuralot}, although an explicit characterization of the effective time-scale ratio is still lacking. From dynamic system (\ref{eq:dynamic_syst}) it follows that for large $\kappa$, $t$ evolves faster. In the case of $\psi$, it evolves faster for small values of $\kappa$.

\subsection{Numerical Results}

Fig.~\ref{fig:flatness} presents NOT results on two one-dimensional Gaussian distributions. The left plot shows the evolution of $\|t-T^\star\|^2_{L_2(\nu)}$, the middle plot shows $\|\nabla\psi-\nabla\psi^\star\|^2_{L_2(\nu)}$, and the right plot shows the values of $F(t,\psi)$. The figure illustrates that convergence of $t_\theta$ may occur while $\psi_\phi$ remains far from the optimal potential $\psi^\star$ (see Sec.~\ref{sec:data} for details on the data used in the experiments).

\begin{figure}[t]
  \centering
  \includegraphics[width=\textwidth]{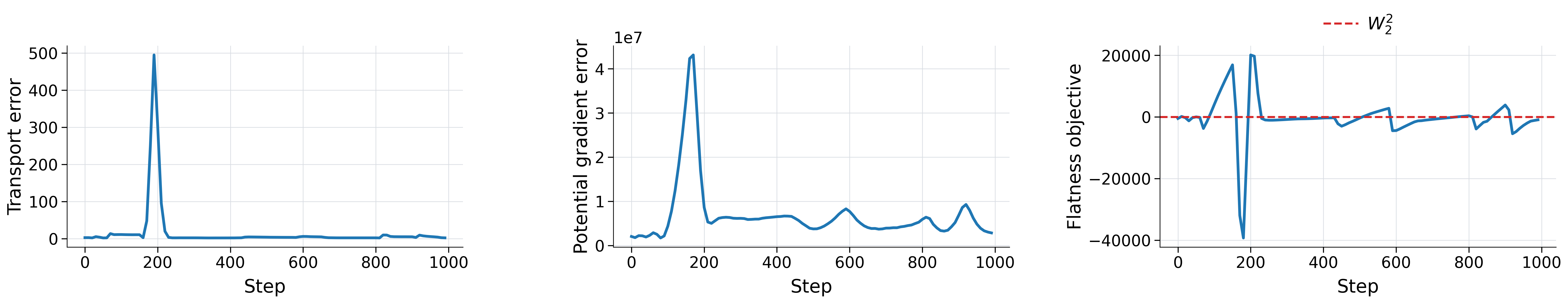} 
  \caption{Convergence of transport map with divergence of potential.}
  \label{fig:flatness}
\end{figure}

We evaluate the performance of NOT in case of two 2D Gaussians with $K\in\{1,\ 5,\ 10,\ 15,\ 20\}$ and $\eta_\psi \in \{\eta_t,\ 0.1\eta_t,\ 0.01\eta_t\}.$ The results in Fig.~\ref{fig:general} demonstrate the importance of the value $K\eta_t/\eta_\psi:$ for large $K$ small $\eta_t/\eta_\psi$ the transport map is approximated better, for small $K$ and large $\eta_t/\eta_\psi$ the potential is learned better. For experiments on more complex data see Sec.~\ref{sec:Experiments}. These examples illustrate that for the approximation of $T^\star$ the potential shouldn't be optimal.


\subsection{Comparison with Previous Results}

Different estimates exist in the literature on NOT, they assume $\lambda$-strong convexity for $v$ like \cite{tarasov2026a}, or even $c$-concavity of $\psi$ like \cite{fan2021neuralmonge_arxiv}. We relax $\psi$ to uniformly Lipschitz functions and in contrast to previous work our estimate does not depend on $\psi$ (the second term of $F(t,\psi)$ can be estimated by $d_{KR}$).

We emphasize that the existence of an infimum does not guarantee the existence of a minimizer (Table~\ref{tab:existence}), a distinction that has led to incorrect claims in the literature, e.g., \cite[Lemma~1]{korotin2023neural}. We also caution against inferring optimality of $t$ and $\psi$ from the condition $F(t,\psi)\approx F(t^\star,\psi^\star)$: as illustrated in Fig.~\ref{fig:flatness}, this implication does not hold in general. 

In \cite{Spurious},  spurious solutions are discussed. 
They rely on the convergence of the semi-dual problem, which  is not guaranteed  due to the degeneracy of the saddle-point (Theorem~\ref{thm:not_clean}). They consider distributions concentrated on $(n-1)$-dimensional manifolds, when  the Monge map dos not exist. Real data usually contain noise, or when their geometry is known, one can use OT on manifolds. 
We argue that spurious solutions in the semi-dual formulation arise  from the degeneracy. Near optimal transport maps, the dual potential becomes non-identifiable: different potentials  yield distinct gradients. As a result, optimization may converge to stationary points that do not correspond to $W_2(\mu,\nu)^2$, a phenomenon known as \emph{gradient deviation} \cite{korotin2021neuralot}.

This viewpoint also explains the empirical observation of \cite{makkuva2020icnn}[Remark~3.4] that unconstrained neural parameterizations of the transport map (e.g., $\nabla v$ in Table~\ref{tab:OT_GAN}) can outperform convex architectures. Imposing structural constraints on the potential, such as convexity or Lipschitz continuity, reduces degeneracy by eliminating flat directions. However, it also alters the optimization landscape by relaxing the exact dual constraints and may prevent perfect separation (\ref{eq:separ}), thereby introducing approximation bias. To make the saddle-point strict, one can add a regularization term on $\psi,$ see, e.g., \cite{korotin2019w2gn}, but it results in a different from OT generator.

\begin{figure}[t]
  \centering
  \includegraphics[width=\textwidth]{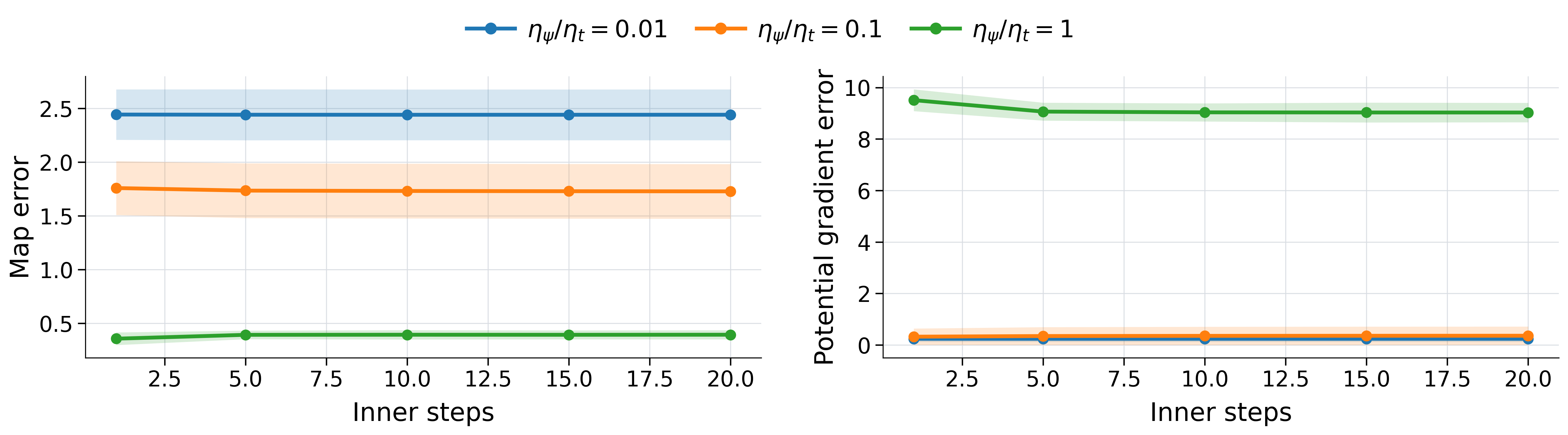}
  \caption{$\|t-T^\star\|^2_{L_2(\mu)}$ and $\|\nabla\psi-\nabla\psi^\star\|^2_{L_2(\nu)}$  with respect to $K$ and $\eta_\psi/\eta_t$.}
  \label{fig:general}
\end{figure}

\subsection{Constrained Optimization Formulation and Connection with Unbalanced OT}
\label{sec:5.5}

Consider  functions $f\colon L_2(X)\to\mathbb{R}$ and $g\colon L_2(X)\to \mathcal{M}(X)$ and the problem
\begin{equation}
        f(t) = \int\limits_X \|x-t(x)\|^2 d\mu \to \min_{t\in L_2(X)}, \qquad    g(t) = t_{\sharp}\mu -\nu = 0. \label{eq:hard_constr}
\end{equation}

Its Lagrangian coincides with $F$ and one could expect the following minimax formulation  
\begin{equation}\label{eq:lagrange_form}
  \inf_{t\in L_2(X)}\sup_{\psi\in C(X)} F(\psi, t).
\end{equation}
In this Monge formulation, $g(t)=0$ is the constrain and $\psi$ is the dual variable, which can be arbitrary, when the constrain is satisfied. The Lagrangian is not convex, and the theory of saddle-point can not be applied (in particular, interchange of infimum and supremum). We consider constrain in week sense, the case of absolutely continuous densities was studied in \cite{Evans}.

Even though problems (\ref{eq:th15}) and (\ref{eq:lagrange_form}) are not equivalent, in the alternating scheme the order of $\sup$ and $\inf$ is interchangeable. Therefore, NOT can be viewed as a constrained optimization problem that makes it similar to unbalanced OT, where the constraint $t_{\sharp}\mu=\nu$ is relaxed via penalization.   

\subsection{Limitations}

Our analysis relies on the existence of an optimal Monge map (see Table~\ref{tab:existence}). In Theorem~\ref{th:conv}, we further assume that $\psi$ is uniformly Lipschitz, which is reasonable in practice—for instance, when neural networks are regularized via gradient clipping or spectral normalization. 

Stronger assumptions are required to characterize the second variation of $F(t,\mu)$ because of its degeneracy. In particular, convexity of the supports of $\mu$ and $\nu$, together with their absolute continuity, plays a crucial role, as demonstrated in \cite{holdermaps}.

Finally, finite-sample approximations impose intrinsic limitations: lower bounds on the accuracy of empirical optimal transport estimators are known (see \cite[Theorem~2]{hutter2021}).

\subsection*{AI use statement}
In this work, we have not used generative AI tools for any of the required-disclosure
tasks: generating synthetic datasets, developing theoretical models or conceptual
frameworks, formulating mathematical claims, providing critical ingredients for
proving mathematical claims, assisting in the writing of proofs, proposing or
refining hypotheses, designing or providing feedback on research methodology or
experiments, implementing methods, assisting with translation, cleaning and
reformatting datasets, supporting qualitative or thematic data analysis, or
interpreting results.

Additionally, we used generative AI tools (ChatGPT) to edit the manuscript for
readability, and to help identify relevant prior work (Claude), format
citations, and suggest keywords for the submission. We have reviewed all
AI-assisted work: all edited text was checked against the original meaning and
technical accuracy by the authors, and all AI-suggested literature references
were independently verified against the original sources before citing. We take
responsibility for the final content of this work, including text, claims, or
artifacts produced with the aid of generative AI.

\subsection*{Reproducibility statement}
An anonymized version of our full codebase, configuration files, and trained
model weights is included as supplementary material, enabling exact
reproduction of all experiments (Sec.~\ref{sec:code}); this will be released
publicly under the MIT license upon acceptance. Section~\ref{sec:data}
gives complete descriptions of the data-generation procedures used to produce
Fig.~\ref{fig:flatness} and Fig.~\ref{fig:general}, and of the ICNN-based
benchmark construction used in our additional experiments (Sec.~\ref{sec:data}),
including all distributional parameters, network architectures, and training
hyperparameters needed to regenerate these datasets from scratch. Complete proofs of Theorems \ref{thm:not_clean} and \ref{th:conv} and Remark~\ref{remk2} are
given in Appendix~\ref{sec:Proofs}.

\subsection*{Acknowledgments}
The authors are grateful to Robert J.~McCann for insightful discussions and for suggesting the key idea underlying the proof of Theorem~\ref{thm:not_clean} and bringing to our attention paper~\cite{refId0}. Also the author is grateful to Mark Bocko for the discussion of the convergence conditions. This work was supported in part by the New York State Center of Excellence in Data Science under Grant C25089A007. The authors declare no competing interests.





\bibliography{refs}
\bibliographystyle{iclr2027_conference}



\appendix


\section{Proofs of Theorems \ref{thm:not_clean} and \ref{th:conv} and Remark~\ref{remk2}}
\label{sec:Proofs}

\begin{proof}[Proof of Theorem~\ref{thm:not_clean}]

\textbf{Step 1: Rewriting the functional.}
We write
\[
F(\psi,t)
=
\int_X \bigl[c(x,t(x)) - \psi(t(x))\bigr]\,d\mu
+
\int_Y \psi(y)\,d\nu.
\]

\medskip

\textbf{Step 2: Minimization over $t$.}
Fix $\psi$. Set
\[
f(x,y) := c(x,y) - \psi(y).
\]
Then
\[
\inf_{t} \int_X f(x,t(x))\, d\mu.
\]
Since the problem is separable in $x$, by a measurable selection argument (e.g., Rockafellar–Wets, Theorem 14.60),
\[
\inf_{t} \int_X f(x,t(x))\, d\mu
=
\int_X \inf_{y \in Y} f(x,y)\, d\mu
=
\int_X \psi^c(x)\, d\mu.
\]
Therefore,
\[
\inf_{t} F(\psi,t)
=
\int_X \psi^c(x)\, d\mu
+
\int_Y \psi(y)\, d\nu.
\]

\medskip

\textbf{Step 3: Proof of (i).}
Taking the supremum over $\psi$, we obtain
\[
\sup_{\psi} \inf_{t} F(\psi,t)
=
\sup_{\psi}
\left(
\int_X \psi^c \, d\mu + \int_Y \psi \, d\nu
\right),
\]
which is exactly the Kantorovich dual problem.

\medskip

\textbf{Step 4: Proof of (ii).}
Using Step 2,
\[
\inf_t F(\psi,t)
=
\int_X \psi^c\, d\mu + \int_Y \psi\, d\nu,
\]
\[
\inf_t F(\psi^{cc},t)
=
\int_X (\psi^{cc})^c\, d\mu + \int_Y \psi^{cc}\, d\nu.
\]
Since $(\psi^{cc})^c = \psi^c$, we obtain
\[
\inf_t F(\psi^{cc},t) - \inf_t F(\psi,t)
=
\int_Y (\psi^{cc} - \psi)\, d\nu \ge 0.
\]
Equality holds if and only if $\psi = \psi^{cc}$ $\nu$-a.e., i.e., $\psi$ is $c$-concave.

\medskip

\textbf{Step 5: Proof of (iii).}
Suppose $t_\#\mu \ne \nu$. Then the signed measure
\[
\sigma := \nu - t_\#\mu
\]
is a nonzero signed Radon measure on the compact metric space $Y$, the Riesz representation theorem guarantees the existence of a continuous function $\tilde{\psi}\in C(Y)$ such that $\int_{Y}\tilde{\psi}\, d\sigma > 0$.
Let $\psi_M = M \tilde{\psi},$ then
\[
F(\psi_M,t)
=
\int c(x,t(x))\, d\mu + M\int \tilde{\psi} \, d\sigma,
\]
which tends to $+\infty$ as $M \to \infty$, we obtain
\[
\sup_{\psi} F(\psi,t) = +\infty.
\]

\medskip

\textbf{Step 6: Proof of (iv).}
Let $\psi^\star$ be an optimal Kantorovich potential. Then $\psi^\star$ is $c$-concave, and there exists an optimal transport map $T^\star$ such that
\[
T^\star(x) \in \arg\min_y \{c(x,y) - \psi^\star(y)\}
\quad \mu\text{-a.e.}
\]
Hence,
\[
c(x,T^\star(x)) - \psi^\star(T^\star(x)) = \psi^{\star c}(x),
\]
and therefore
\[
F(\psi^\star,T^\star)
=
\int_X \psi^{\star c}(x)\, d\mu
+
\int_Y \psi^\star(y)\, d\nu.
\]

For any $t$, using $\psi^{\star c}(x) \le c(x,t(x)) - \psi^\star(t(x))$, we obtain
\[
F(\psi^\star,t)
\ge
F(\psi^\star,T^\star).
\]

If $t = T^\star$, then $t_\#\mu = \nu$, and thus
\[
F(\psi, T^\star)
=
\int c(x,T^\star(x))\, d\mu,
\]
which is independent of $\psi$. Hence
\[
F(\psi^\star,T^\star) = F(\psi,T^\star)
\quad \text{for all } \psi.
\]

Formula (\ref{eq:map_via_grad}) follows from \cite{Gangbo_1996}.

\medskip

\textbf{Step 7: Proof of (v).}
Fix $\psi^\star$. Then minimizing $t \mapsto F(\psi^\star,t)$ is equivalent to minimizing
\[
\int_X \bigl[c(x,t(x)) - \psi^\star(t(x))\bigr]\, d\mu.
\]
By Step 2, any minimizer satisfies
\[
t(x) \in \arg\min_y \{c(x,y) - \psi^\star(y)\}
\quad \mu\text{-a.e.}
\]
For the quadratic cost, this minimizer is unique $\mu$-almost everywhere (by Brenier’s theorem), hence $t = T^\star$ $\mu$-a.e. and the minimizer is unique.

\end{proof}

\begin{proof}[Proof of Remark~\ref{remk2}]
\textbf{Step 1.} Use the pointwise inequality
\begin{equation}
\begin{split}
       |c(x,\, t_k(x))-c(x,\, T^\star(x))| &= \left|\frac{1}{2}|x-t_k(x)|^2 - \frac{1}{2}|x-T^\star(x)| \right| \\
       &= \frac{1}{2}\left| \langle t_k(x)-T^\star(x),\, t_k(x) + T^\star(x) -2x\rangle\right| \\
        &\leq \frac{1}{2} |t_k(x)-T^\star(x)|\, |t_k(x) + T^\star(x) -2x|. 
\end{split}
\end{equation}
Since $X$ and $Y$ are  bounded, there exists $C>0$ such that
$|t_k(x)+T^\star(x)-2x| \leq 2C$
for every $x\in X$. Hence, by the Cauchy--Schwarz inequality,
\begin{equation}
          \left|\int c(x,\, t_k(x))\,d\mu(x) - W_2(\mu,\, \nu)^2\right|
        \leq
        C \int |t_k-T^\star|\,d\mu(x) \leq C \|t_k -T^\star\|_{L^2(\mu)},
\end{equation}
where we used the fact that $\mu$ is a probability measure.

\medskip

\textbf{Step 2.} Since $T^\star_{\sharp}\mu=\nu,$  consider the coupling
$
    \pi_k = (t_k,\, T^\star)_{\sharp}\mu.
$
By the definition, 
\begin{equation}
    W_2(t_{k\sharp}\mu,\, \nu)^2
    \leq
    \int |t_k(x)-T^\star(x)|\,d\mu(x) = \|t_k-T^\star\|^2_{L^2(\mu)}.
\end{equation}
Therefore,
\begin{equation}
    d_{KR}(t_{k\sharp}\mu,\, \nu)
    \leq W_1(t_{k\sharp}\mu,\, \nu) \leq W_2(t_{k\sharp}\mu,\, \nu)
    \leq
    \|t_k-T^\star\|_{L^2(\mu)}.
\end{equation}

\medskip

\textbf{Step 3.} Set
$
    g_k = \psi_k-\psi^\star 
$ and
$c_k = (\max g_k + \min g_k)/2.$ Then
\begin{equation}
\begin{split}
    \left| \int(\psi_k-\psi^\star)\,d(t_{k\sharp}\mu -\nu) \right| &= \left| \int(g_k-c_k)\,d(t_{k\sharp}\mu -\nu) \right| \\ &= M \left| \int \frac{g_k-c_k}{M}\,d(t_{k\sharp}\mu -\nu) \right| \leq M d_{KR}(t_{k\sharp},\, \nu),
\end{split}
\end{equation}
where $M=\max(2L',\, \operatorname*{diam}(Y)L').$ By $L'$ we denote a common Lipschitz constant for $\psi_k$ and $\psi^\star$.

\medskip

\textbf{Step 4.} From
\begin{equation}
    F(\psi_k,\, t_k) - W_2(\mu,\, \nu)^2 =\int c(x,\, t_k(x))\, d\mu - W_2(\mu,\, \nu)^2 + \int (\psi_k-\psi^\star)\, d(t_{k\sharp}\mu-\nu)
\end{equation}
we have
\begin{equation}
\begin{split}
    \left| F(\psi_k,\, t_k) - W_2(\mu,\, \nu)^2\right| &\leq \left| \int c(x,\, t_k(x))\, d\mu - W_2(\mu,\, \nu)^2 \right| + \left| \int (\psi_k-\psi^\star)\, d(t_{k\sharp}\mu-\nu) \right| \\ &\leq (C+M)\|t_k-T^\star\|_{L^2(\mu)}.
\end{split}
\end{equation}

Altogether, this gives that 
   $ F(\psi_k,\, t_k) \to W_2(\mu,\, \nu)^2$  in $L^2(\mu)$ and $d_{KR}(t_{k\sharp}\mu,\, \nu)\to 0.$
\end{proof}

\begin{proof}[Proof of Theorem~\ref{th:conv}] First, note that $\psi^\star\in \operatorname*{Lip}_L$ for some $L>0,$ because it is differentiable. Therefore, class $\psi\in C(Y)$ can be narrowed to $\psi\in\operatorname*{Lip}_L$. Recall that $T^\star(x)=\nabla u(x)$ and $S^\star(y)=\nabla v(y),$ where
\begin{equation}
    u(x) = \frac{|x|^2}{2} - \varphi^\star(x),\qquad v(y) = \frac{|y|^2}{2} - \psi^\star(y).
\end{equation}
\textbf{Step 1.} \cite[Theorem~3.3]{Figalli_2014} guaranties $u\in C^2(\overline{X})$ that gives the estimate $\|u\|_{_{\infty}}< C$ for some $C>0.$ Then, from the boundedness of $\det D^2 u$  away from zero and infinity, the positive definetness of the Hessian follows  $D^2 v\geq \lambda I$ (see also \cite[Corollary~3.2]{Gigli_2011}).\footnote{They prove $\lambda I\leq D^2 u\leq \Lambda I,$ we note that $D^2v = (D^2u)^{-1}$ and redenote $1/\Lambda$ by $\lambda.$} Therefore, $v$ is $\lambda$-strongly convex, i.e.,
\begin{equation}\label{eq:proof_strong_conv}
    v(y_2)\geq v(y_1) + \langle \nabla v(y_1), y_2-y_1\rangle +\frac{\lambda}{2} |y_2 - y_1|^2,\qquad \forall y_1,y_2\in Y.
\end{equation}

\textbf{Step 2.} 

\begin{equation}\label{eq:proof_conv_step2}
    \begin{split}
        F(t,\psi^\star)-F(T^\star,\psi^\star) =& \frac{1}{2}\int |x-t(x)|^2d\mu - \int\psi^\star(t(x))d\mu + \int \psi^\star(y)d\nu\\
        &-\frac{1}{2}\int |x-T^\star(x)|^2d\mu + \int\psi^\star(T^\star(x))d\mu - \int \psi^\star(y)d\nu\\
        =& \int\left[ -\langle x,t-T^\star\rangle+v(t(x))-v(T^\star(x))\right] d\mu\\
        \overset{(\ref{eq:proof_strong_conv})}{\geq}& \int\left[ -\langle x,t-T^\star\rangle+\langle \nabla v(T^\star(x)), t-T^\star\rangle +\frac{\lambda}{2}|t-T^\star(x)|^2\right] d\mu\\
        =& \frac{\lambda}{2} \|t-T^\star\|_{L^2(\mu)}^2.
    \end{split}
\end{equation}
We uses the fact that $\nabla v(T^\star(x))=x.$

\textbf{Step 3.}
We need the following estimate:
\begin{equation}\label{A:dkl_estimate}
    \int_Y (\psi^\star-\psi)  d(t_{\sharp}\mu - \nu) \le M  d_{KR}(t_{\sharp}\mu, \nu)
\end{equation}
with $M = \max(2L',\, L' \cdot \mathrm{diam}(Y))$, where $L' = \max\{L,\, \|\psi^{\star}\|_{Lip}\}$ is a common Lipschitz constant. 

For each $\psi$ define $g = \psi^\star-\psi$. Since $\psi, \psi^\star$ are $L'$-Lipschitz, $g$ is $2L'$-Lipschitz. Since $t_{\sharp}\mu - \nu$ has total mass zero, 
\begin{equation}
\int (g-c)  d(t_{k\sharp}\mu - \nu) = \int g  d(t_{k\sharp}\mu - \nu)    
\end{equation}
for any constant $c$. We take $c = \tfrac{1}{2}(\max g + \min g)$. With this choice, 
\begin{equation}
\|g - c\|_\infty \le \tfrac{1}{2}(\max g - \min g) \le \tfrac{1}{2} \cdot 2L' \cdot \mathrm{diam}(Y) = L' \cdot \mathrm{diam}(Y).    
\end{equation}

The function $(g - c)/M$ satisfies both  $|(g-c)/M| \le 1$ and is 1-Lipschitz, so it is admissible in the supremum defining $d_{KR}$. Hence 
\begin{equation}
 \int_Y (\psi^\star-\psi)  d(t_{\sharp}\mu - \nu) = M \int \frac{g - c}{M}  d(t_{\sharp}\mu - \nu) \le M  d_{KR}(t_{\sharp}\mu, \nu).   
\end{equation}



\textbf{Step 4.} Notice that
\begin{equation}\label{eq:append_f}
    F(t,\psi) = F(t,\psi^\star) + \int(\psi-\psi^\star) d(t_\sharp\mu-\nu),
\end{equation}
then
\begin{equation}
\begin{split}
    \frac{\lambda}{2} \|t-T^\star\|_{L^2(\mu)}^2 &\overset{(\ref{eq:proof_conv_step2})}{\leq} F(t,\psi^\star)-F(T^\star,\psi^\star) \\
    &\overset{(\ref{eq:append_f})}{=} F(t,\psi) - F(T^\star,\psi^\star) - \int(\psi-\psi^\star) d(t_\sharp\mu-\nu)\\
    &\overset{(\ref{A:dkl_estimate})}{\leq} F(t,\psi) - W_2(\mu,\nu)^2 + Md_{KR}(t_\sharp\mu,\nu).
\end{split}
\end{equation}
Because $\lambda>0$ and does not depend on $t$ or $\psi,$ we can divide the inequality by $\lambda/2$ to obtain the final result. 
\end{proof}

\section{Alternative Formulation of NOT in Case of Quadratic Cost}
\label{sec:app_mcf}

If we look at the Kantorovich problem in case of quadratic cost
\begin{equation}
    \inf_{\pi\in\Pi(\mu,\nu)}\int\limits_{X\times Y}\left(\frac{|x|^2}{2} +\frac{|y|^2}{2} -\langle x,y\rangle\right)d\pi = \int\limits_X\frac{\|x\|^2}{2}d\mu+\int\limits_Y\frac{\|y\|^2}{2}d\nu - \sup_{\pi\in\Pi(\mu,\nu)}\int\limits_{X\times Y}\langle x,y\rangle d\pi
\end{equation}
we can connect the following maximum correlation functional to the Monge formulation:
\begin{equation}
    \sup_{t\in T(\mu,\nu)}\int\limits_{X}\langle x,t(x)\rangle d\mu. 
\end{equation}

Therefore, by duality,
\begin{equation}
\int\limits_X\frac{\|x\|^2}{2}d\mu+\int\limits_Y\frac{\|y\|^2}{2}d\nu - \sup_{t\in T(\mu,\nu)}\int\limits_{X}(x,t(x))d\mu =    \underset{\varphi(x)+\psi(y)\leq\|x-y\|^2}{\sup}\left(\int\limits_{X}\varphi d\mu + \int\limits_{Y}\psi d\nu\right),
\end{equation}
or
\begin{equation}
MCF(\mu,\nu) =    \inf_{u(x)+v(y)\geq\langle x,y\rangle}\left(\int\limits_{X}u d\mu + \int\limits_{Y}v d\nu\right).
\end{equation}
where $u(x)=\frac{|w|^2}{2}-\phi(x)$ and $v(y)=\frac{|y|^2}{2}-\psi(y)$ are convex functions.

Using semi-duality, we have
\begin{align}
MCF(\mu,\nu) =  
    \inf_{\substack{u \in C(Y)\\ u\text{ convex}}} 
\left( \int_X u(x)\, d\mu + \int_Y \sup_{x\in X}(\langle x,y\rangle -u(x))\, d\nu \right) \\
=\inf_{\substack{v \in C(Y)\\ v\text{ convex}}} 
\left( \int_X \sup_{y\in Y}(\langle x,y\rangle -v(x))\, d\mu + \int_Y g\, d\nu \right).
\end{align}
And using interchange theorem,
\begin{align}
MCF(\mu,\nu) =  
    \inf_{\substack{u \in C(X)\\ u\text{ convex}}} \sup_{s\in\mathcal{M}(Y,X)}\Bigl( \int\limits_X u(x) d\mu + \int\limits_Y  \bigl[ \langle s,y\rangle -u(s(y))\bigr]d\nu \Bigr)\\
    =\inf_{\substack{v \in C(Y)\\ v\text{ convex}}} \sup_{t\in\mathcal{M}(Y,X)}\Bigl( \int\limits_X \bigl[ \langle x,t\rangle -v(t(x))\bigr] d\mu + \int\limits_Y  v(y) d\nu \Bigr).
\end{align} 

Note that
\begin{equation}
    SDP(\mu,\nu) = \int\limits_X\frac{\|x\|^2}{2}d\mu+\int\limits_Y\frac{\|y\|^2}{2}d\nu - MCF(\mu,\nu).
\end{equation}

Consider the functional
\begin{equation}
    \tilde{F}(t,v) = \int\limits_X  \langle x,t(x)\rangle d\mu + \int\limits_Y  v(y) d(\nu-t_\sharp\mu).
\end{equation}
The same analysis can be applied as to $F(t,\psi).$ 







\section{Code Availability}
\label{sec:code}
Upon acceptance, we will release the full codebase, trained weights, and generated data under the MIT license on our GitHub repository and project page. Due to double-blind anonymity constraints, we do not include the public repository link at submission time; instead, the submitted supplementary material contains the code needed to reproduce this work.

\section{Data Generation}
\label{sec:data}

\paragraph{Figures in Sec.~\ref{sec:discussion}}
To obtain Fig.~\ref{fig:flatness}, we consider one-dimensional Gaussian distributions $\mathcal{N}(0,1)$ and $\mathcal{N}(2,1.5^2)$. Both $t$ and $\psi$ are parameterized by MLPs and trained until convergence. We then perturb $\psi$ with noise and continue training for an additional 1000 steps. For Fig.~\ref{fig:general}, we consider two-dimensional Gaussians $\mathcal{N}(0,I)$ and $\mathcal{N}(\mathbf{1}, 2^2 I)$, and run NOT with 30 different random seeds. The corresponding models and trained weights are available on the project page.

\subsection{Additional Experiments}
We use an ICNN-generated continuous OT benchmark, following the analytic-map
benchmark idea of \citet{korotin2021neuralot}. The source distribution
$\mu\subset\mathbb{R}^{192}$ is a 24-component Gaussian mixture:
\begin{equation}
    X\sim \sum_{j=1}^{24}\pi_j
    \mathcal{N}\!\left(m_j,\;0.35^2\operatorname{diag}(s_j^2)\right).
\end{equation}
The component means $m_j$ are random directions normalized to radius $2.5$.
Mixture weights are unequal, with $\pi=\operatorname{softmax}(\ell)$ and
$\ell_j\sim\mathcal{N}(0,1.25^2)$. Each component has anisotropic diagonal scale
$s_j$, where $\log s_{jk}$ is sampled with standard deviation $0.9$ and centered
within each component.

The ground-truth Brenier potential is a randomly initialized input-convex neural
network $\Phi_0:\mathbb{R}^{192}\to\mathbb{R}$ 
with five hidden layers of width 192, softplus activations, initialization
standard deviation $0.14$, and strong-convexity coefficient $0.05$. We calibrate
its scale on 4096 source samples:
\begin{equation}
    \Phi = a\Phi_0, \qquad
    a =
    \frac{\sqrt{\mathbb{E}\|X\|^2}}
         {\sqrt{\mathbb{E}\|\nabla \Phi_0(X)\|^2}} .
\end{equation}
The target distribution is then defined by the exact pushforward:
\begin{equation}
    T^\star(x)=\nabla \Phi(x), \qquad \nu=T^\star_{\#}\mu .
\end{equation}
Training batches are resampled every epoch with batch size 512 and 128 batches
per epoch; validation and test splits contain 8192 samples each. In the timescale
sweep we use $K\in\{1,2,5,10,20\}$, ratios
$\eta_\psi/\eta_T\in\{0.02,0.05,0.1,0.25,0.5,1\}$,
$\eta_T=5\cdot10^{-4}$, 4096 outer iterations, and disable solver noise.

\paragraph{Target-potential diagnostics.}
\label{sec:target_potential_diagnostics}
Because the dataset is generated by a known convex potential $\Phi$, we also know
the canonical target-side potential on paired samples. For
$y=T^\star(x)=\nabla\Phi(x)$, the convex conjugate satisfies\footnote{To avoid notational ambiguity, we use $*$ to denote the Legendre--Fenchel transform, while $\star$ denotes the optimal transport map and potential.}
\begin{equation}
    \Phi^\ast(y)=\langle x,y\rangle-\Phi(x),
    \qquad \nabla\Phi^\ast(y)=x .
\end{equation}
Thus, for max-correlation-style methods we compare the learned target potential
$v$ to $v^\star=\Phi^\ast$. For the semi-dual quadratic-cost potential used by
OTP, the corresponding $c$-concave potential is
\begin{equation}
    \psi^\star(y)=\frac12\|y\|^2-\Phi^\ast(y),
    \qquad \nabla\psi^\star(y)=y-x .
\end{equation}
Potential values are compared only after centering, since potentials are
identified up to an additive constant. We report centered target-potential MSE
and target-potential gradient MSE on 512 validation pairs.

\section{Experiments}
\label{sec:Experiments}



\subsection{OTP}
OTP is the direct-map semi-dual NOT solver: an MLP map $T_\theta$ is trained
against a $c$-concave DenseICNN target potential $\psi_\phi$ using the quadratic
semi-dual objective \citep{korotin2022neural,rout2021generative}. In the sweep,
$K$ is the number of map updates per potential update. The results are shown in Figure~\ref{fig:errors_otp}.

\begin{figure}[h]
    \centering

    \begin{subfigure}{0.48\linewidth}
        \centering
        \includegraphics[width=\linewidth]{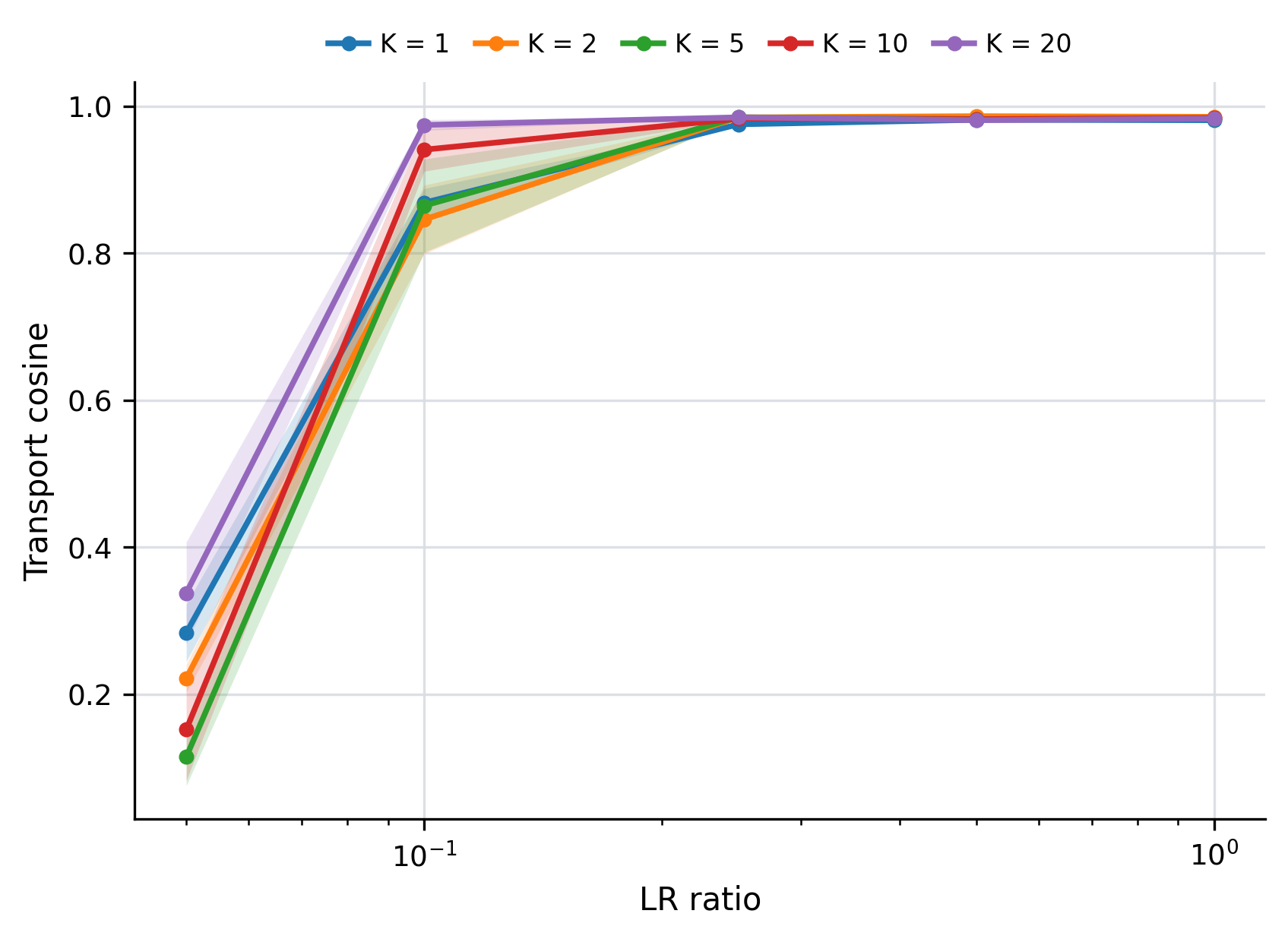}
        \caption{Map cosine similarity.}
        \label{fig:map_error}
    \end{subfigure}
    \hfill
    \begin{subfigure}{0.48\linewidth}
        \centering
        \includegraphics[width=\linewidth]{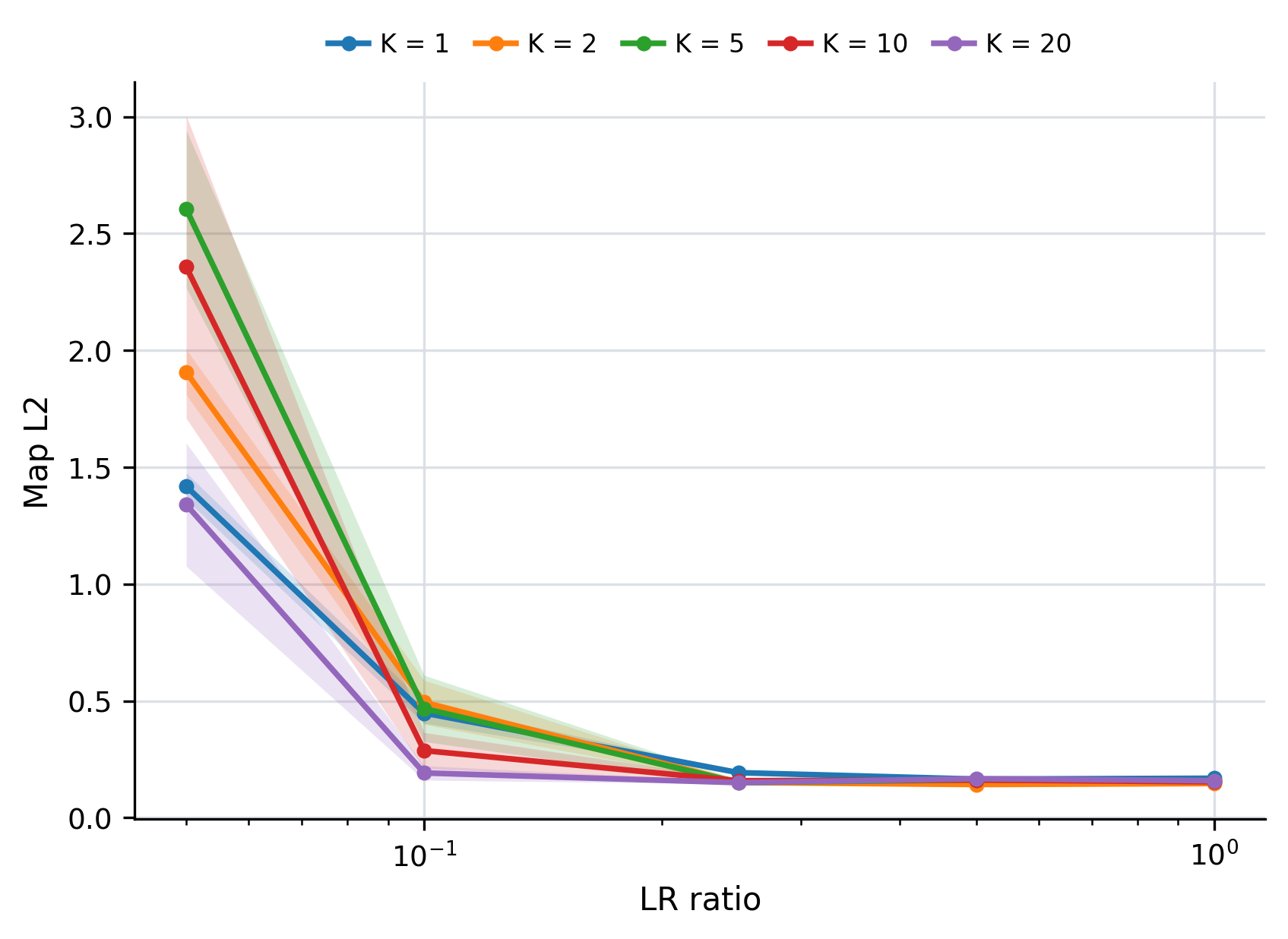}
        \caption{Map error.}
        \label{fig:grad_error}
    \end{subfigure}\\

    \begin{subfigure}{0.48\linewidth}
        \centering
        \includegraphics[width=\linewidth]{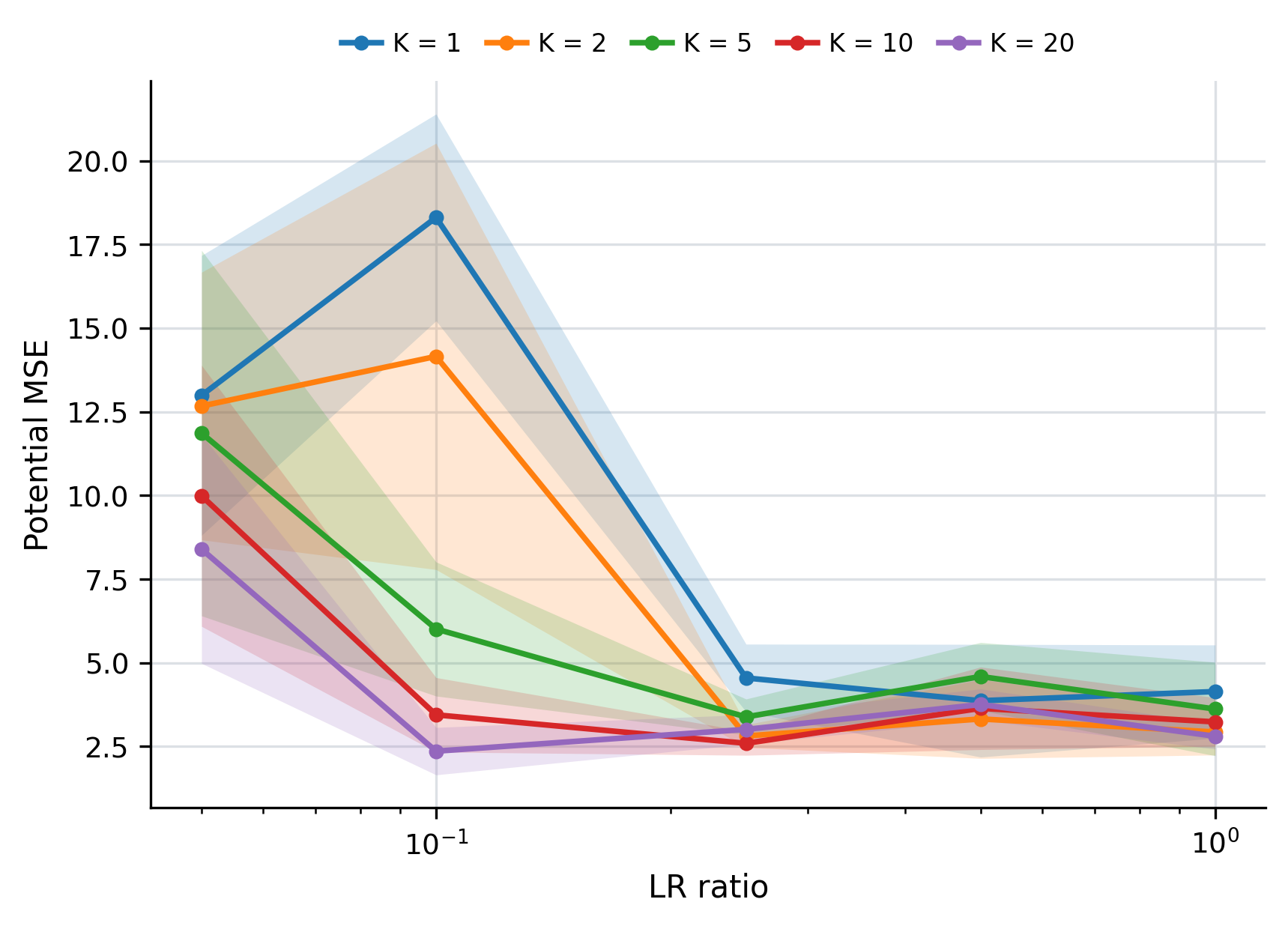}
        \caption{Potential error.}
        \label{fig:map_error}
    \end{subfigure}
    \hfill
    \begin{subfigure}{0.48\linewidth}
        \centering
        \includegraphics[width=\linewidth]{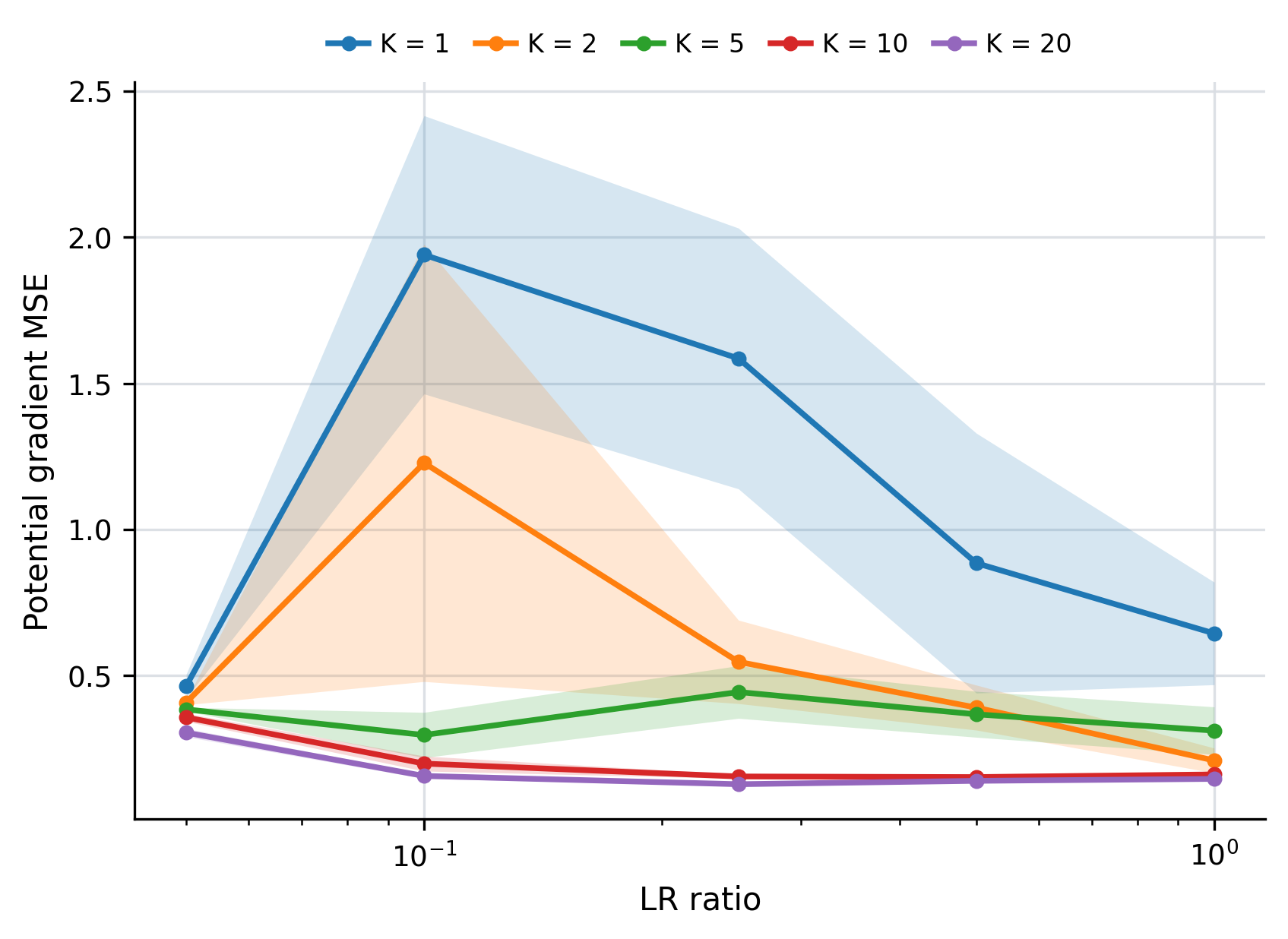}
        \caption{Potential gradient error.}
        \label{fig:grad_error}
    \end{subfigure}\\

    \begin{subfigure}{0.48\linewidth}
        \centering
        \includegraphics[width=\linewidth]{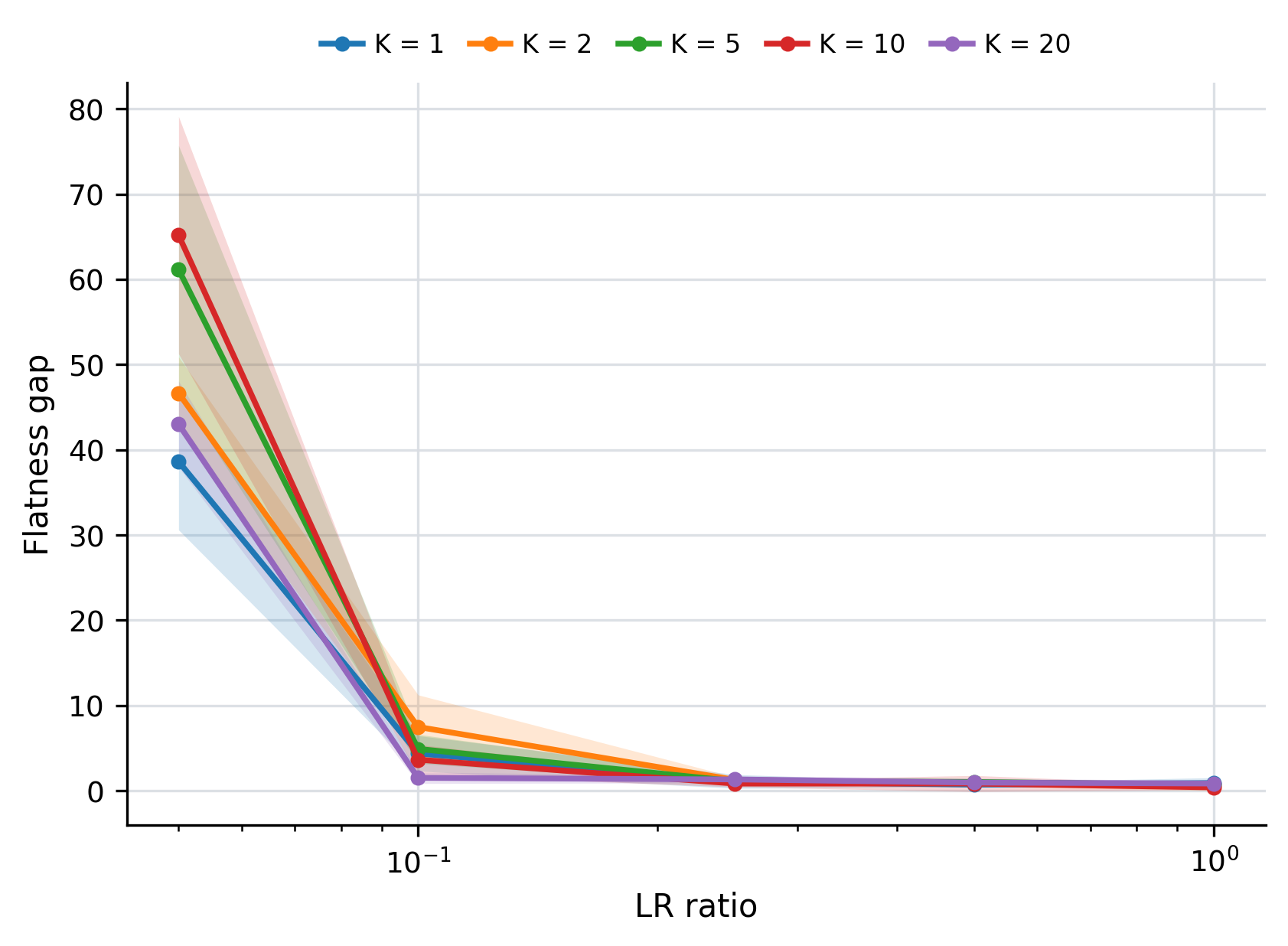}
        \caption{Flatness.}
        \label{fig:map_error}
    \end{subfigure}
    \hfill
    \begin{subfigure}{0.48\linewidth}
        \centering
        \includegraphics[width=\linewidth]{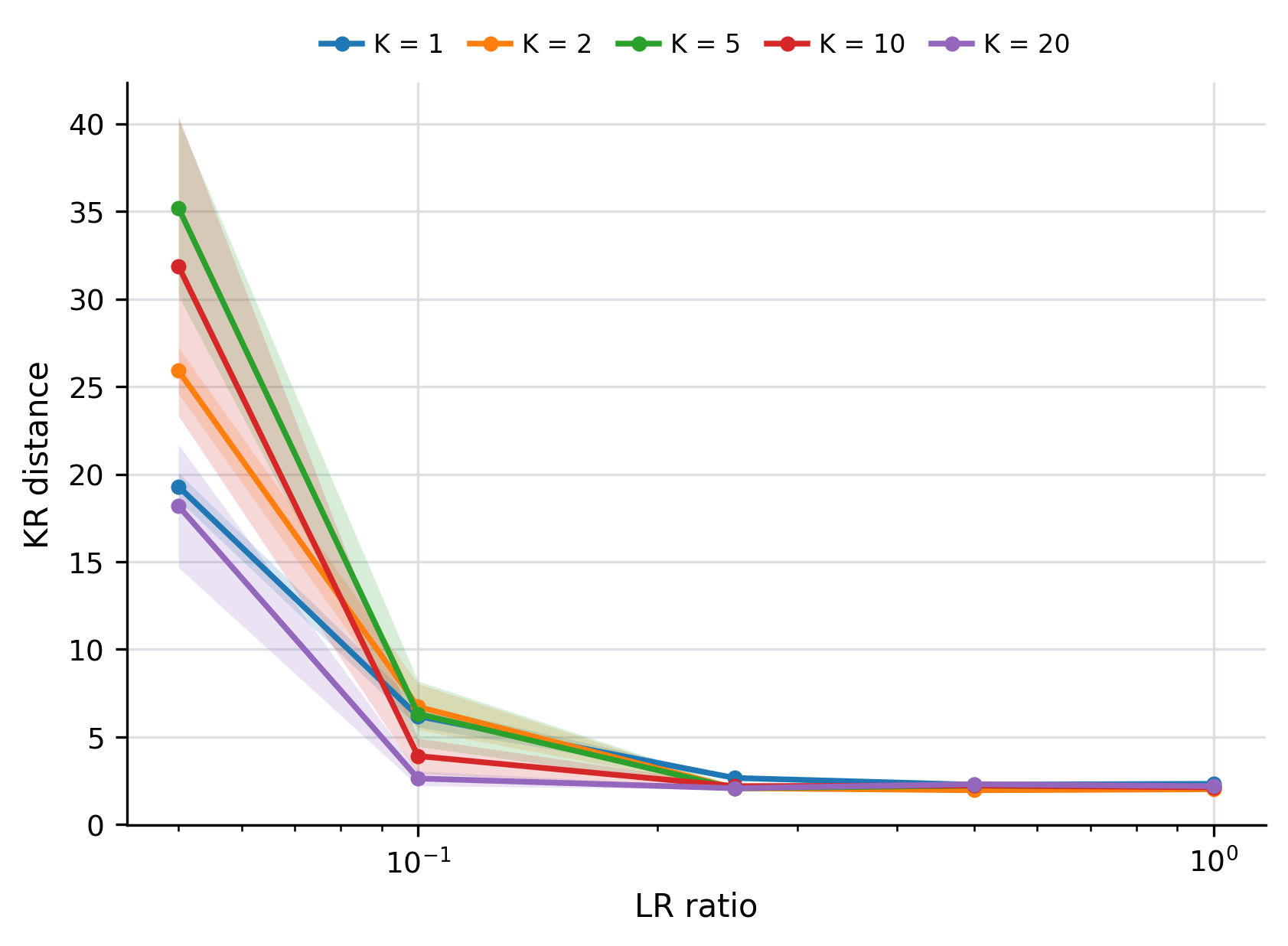}
        \caption{$d_{KR}$.}
        \label{fig:grad_error}
    \end{subfigure}

    \caption{Convergence behavior of the transport map and potential for the OTP method.}
    \label{fig:errors_otp}
\end{figure}

\subsection{Monge Map}
MongeMap follows the weak-form neural Monge-map formulation of
\citet{fan2023neural}. It uses a direct MLP map and an unconstrained MLP
potential, trained with the same alternating map/potential schedule. The results are shown in Figure~\ref{fig:errors_mm}.

\begin{figure}[h]
    \centering

    \begin{subfigure}{0.48\linewidth}
        \centering
        \includegraphics[width=\linewidth]{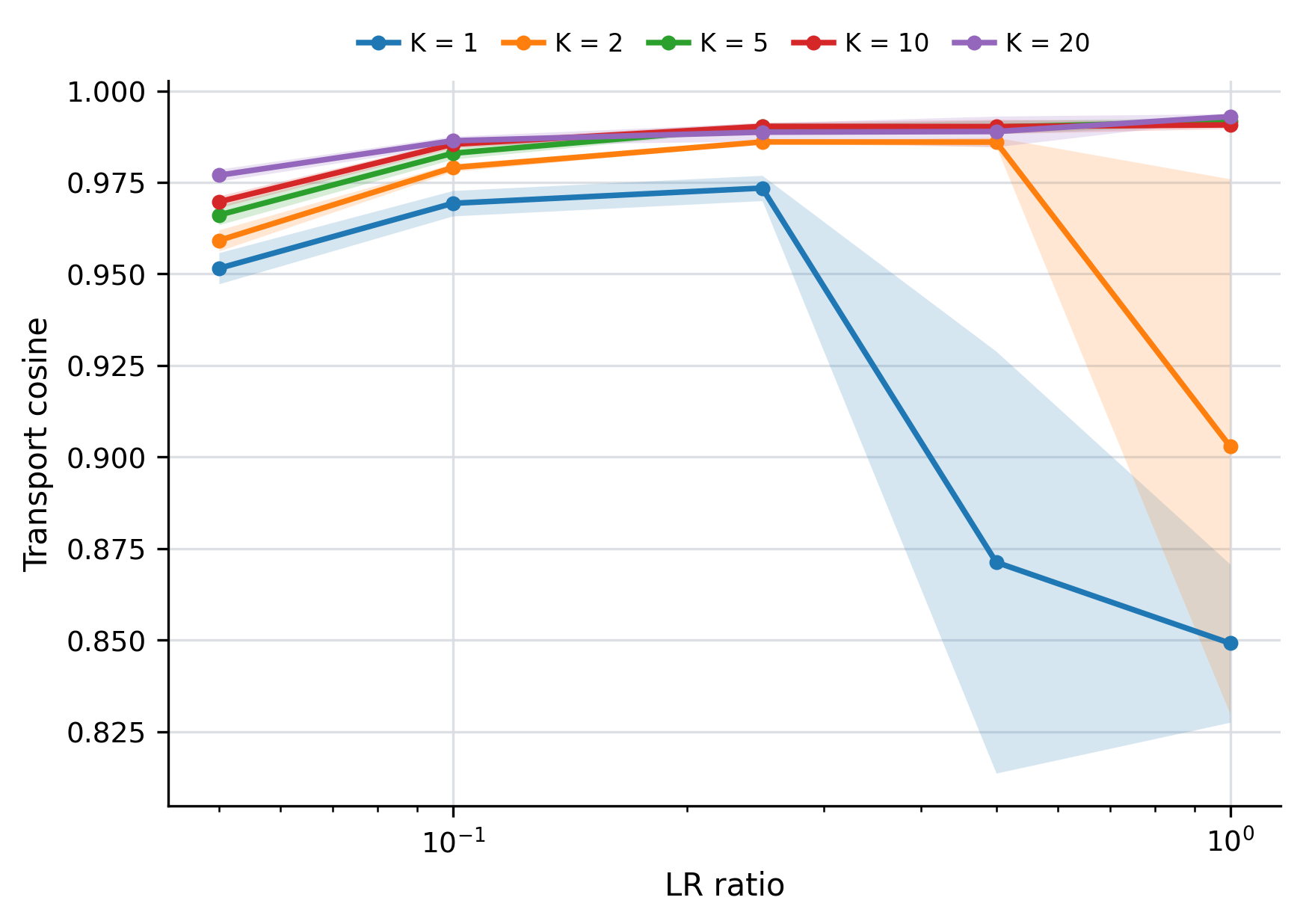}
        \caption{Map cosine similarity.}
        \label{fig:map_error}
    \end{subfigure}
    \hfill
    \begin{subfigure}{0.48\linewidth}
        \centering
        \includegraphics[width=\linewidth]{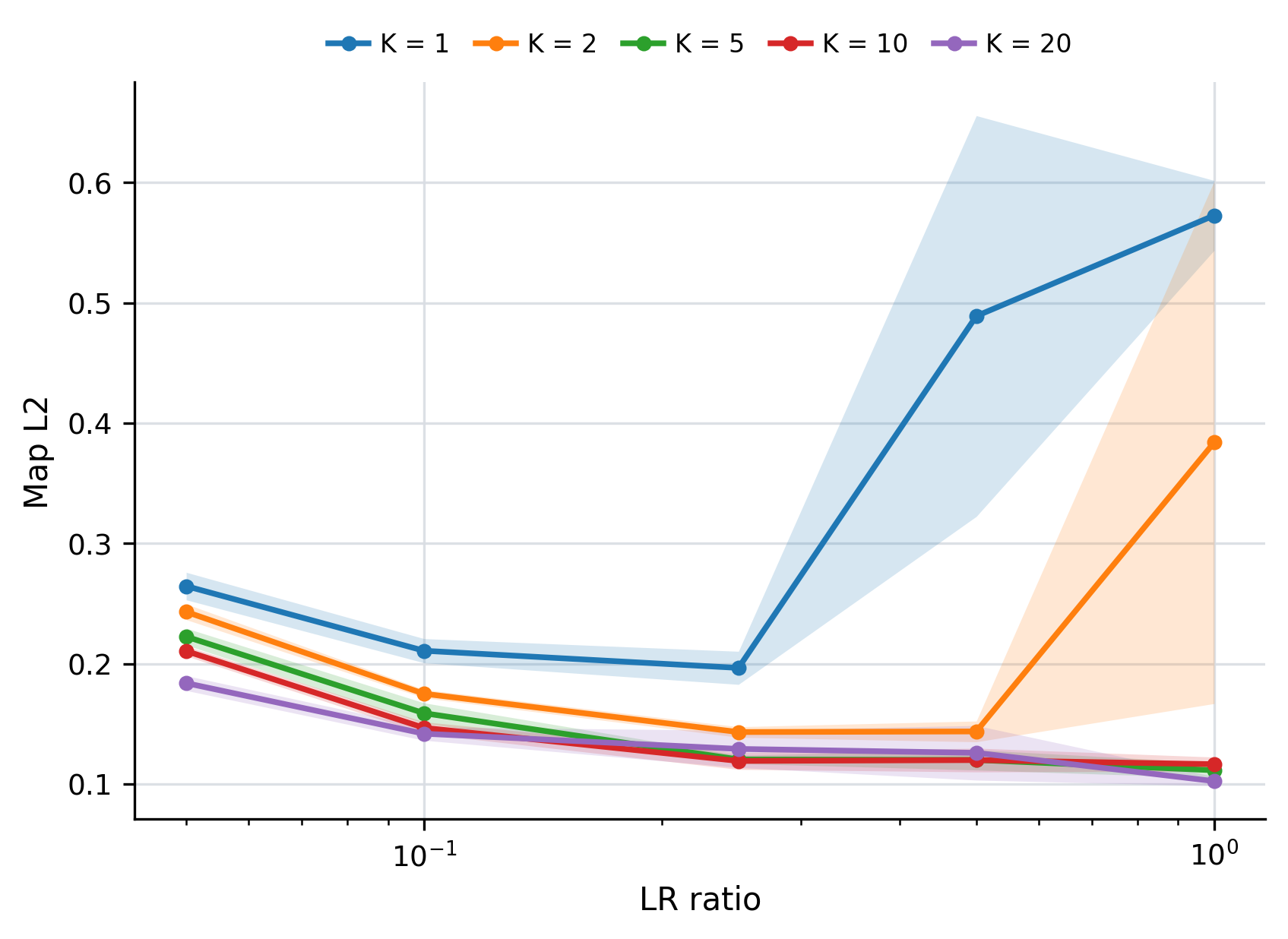}
        \caption{Map error.}
        \label{fig:grad_error}
    \end{subfigure}\\

    \begin{subfigure}{0.48\linewidth}
        \centering
        \includegraphics[width=\linewidth]{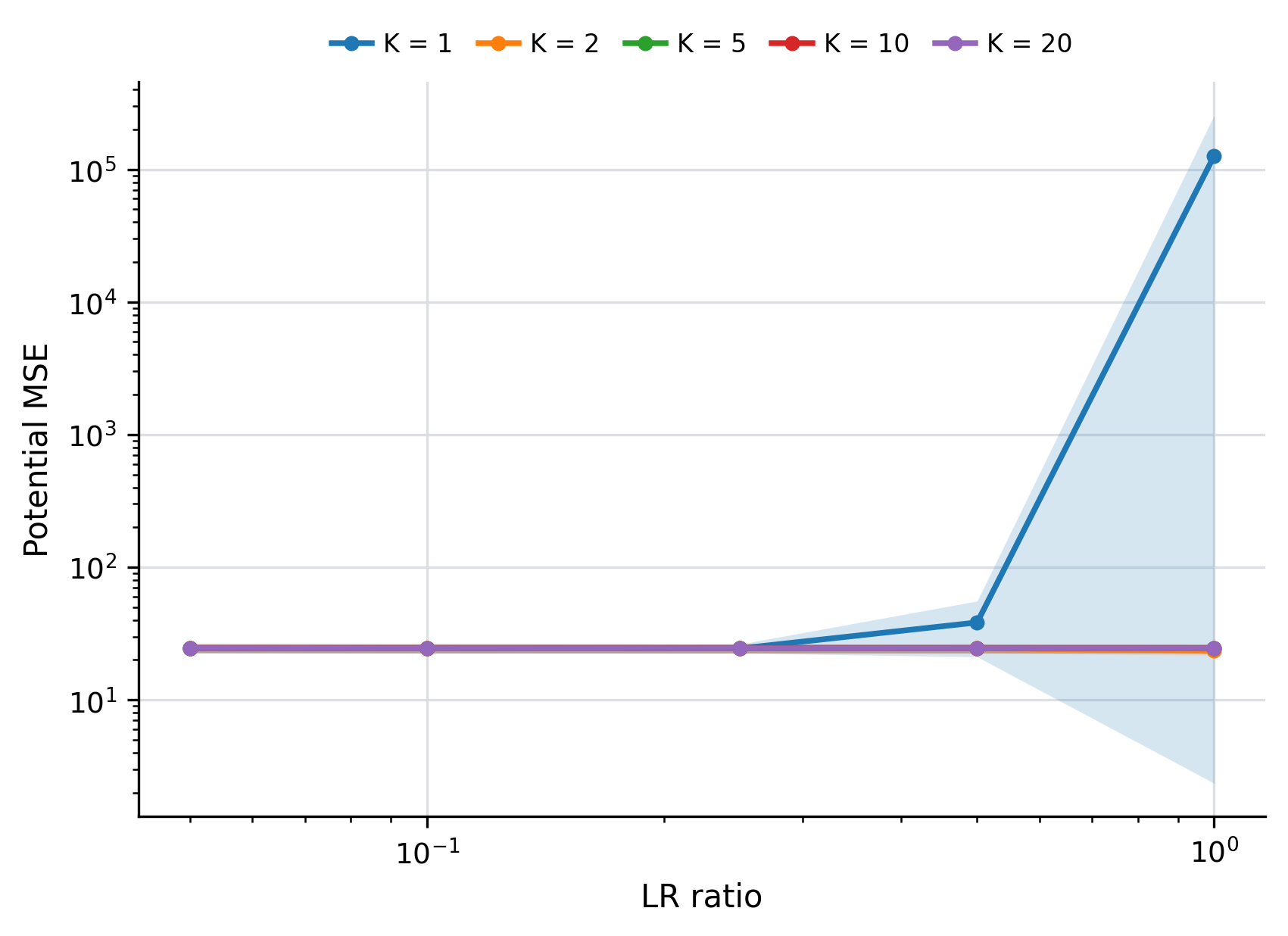}
        \caption{Potential error.}
        \label{fig:map_error}
    \end{subfigure}
    \hfill
    \begin{subfigure}{0.48\linewidth}
        \centering
        \includegraphics[width=\linewidth]{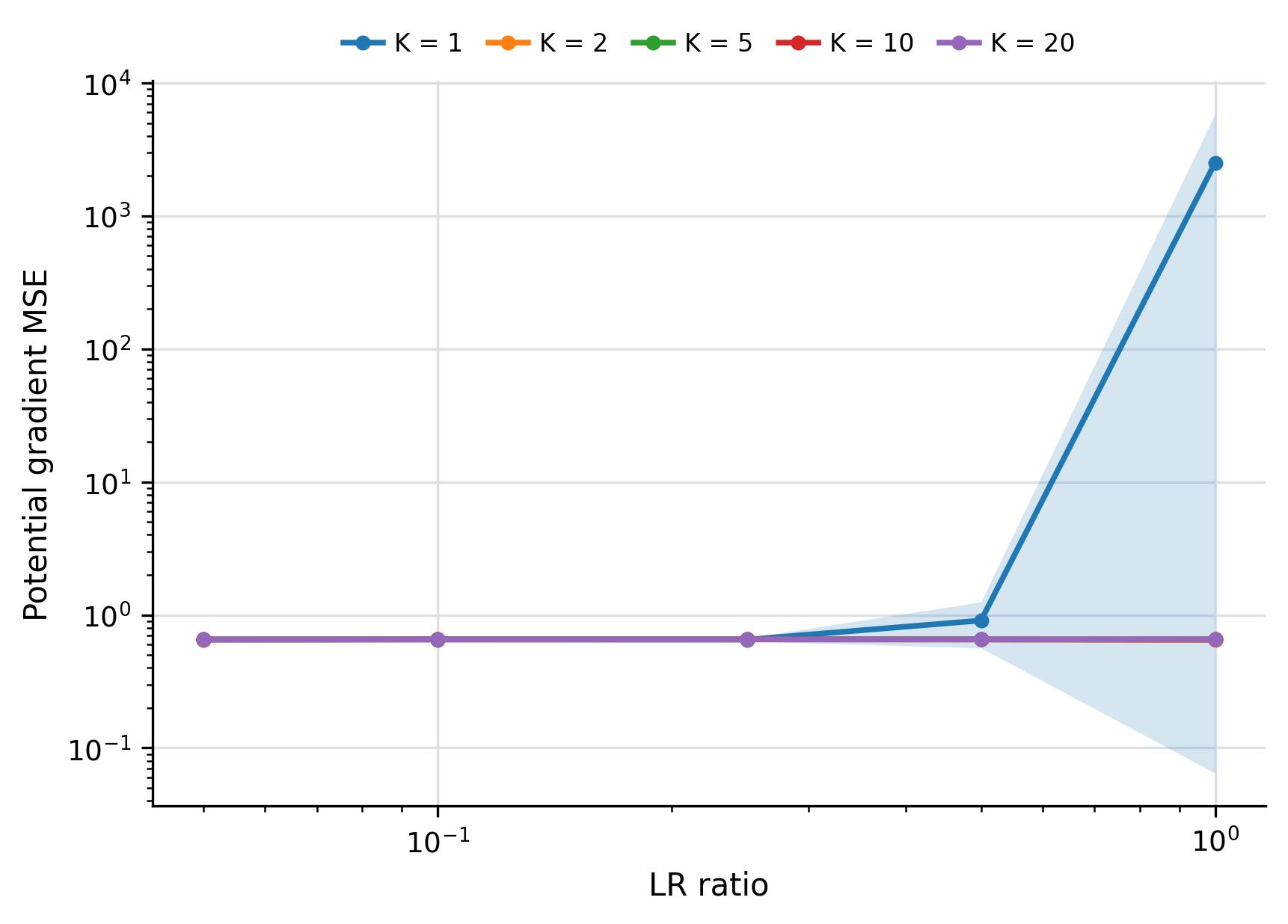}
        \caption{Potential gradient error.}
        \label{fig:grad_error}
    \end{subfigure}\\

    \begin{subfigure}{0.48\linewidth}
        \centering
        \includegraphics[width=\linewidth]{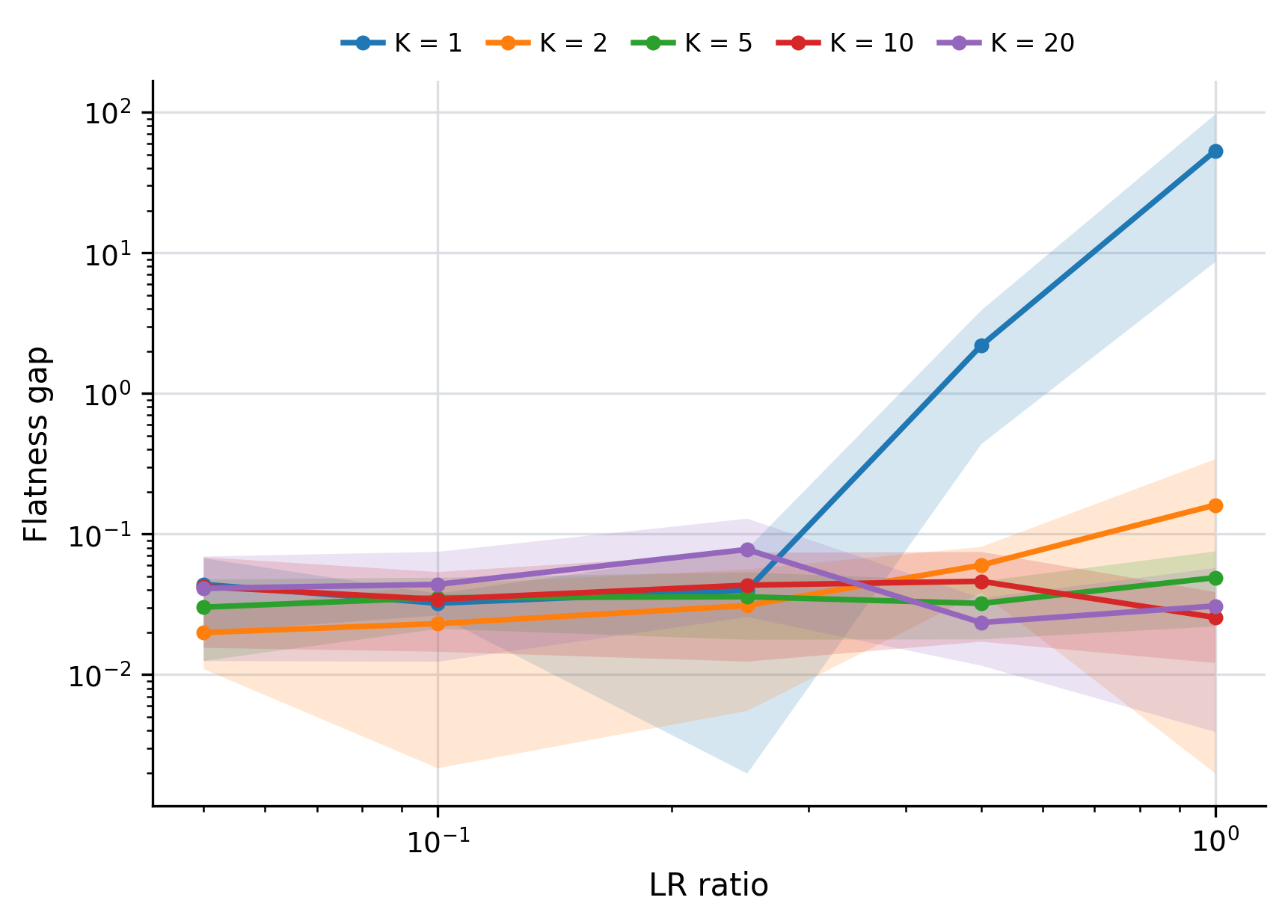}
        \caption{Flatness.}
        \label{fig:map_error}
    \end{subfigure}
    \hfill
    \begin{subfigure}{0.48\linewidth}
        \centering
        \includegraphics[width=\linewidth]{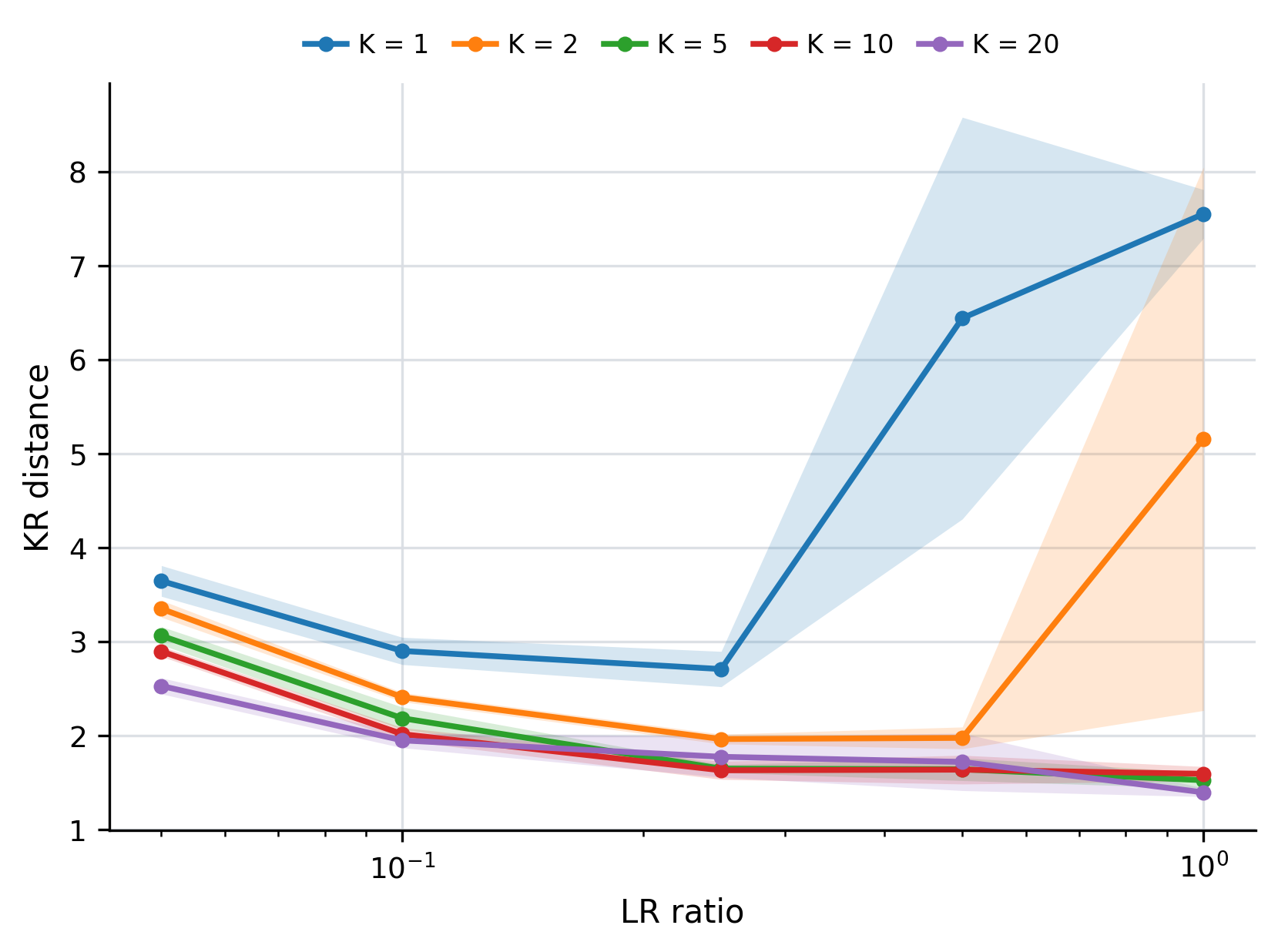}
        \caption{$d_{KR}$.}
        \label{fig:grad_error}
    \end{subfigure}

    \caption{Convergence behavior of the transport map and potential for the Monge Map method.}
    \label{fig:errors_mm}
\end{figure}

\subsection{Max Correlation}
MaxCorr uses the maximum-correlation form of quadratic OT:
\begin{equation}
    \langle x,T(x)\rangle-g(T(x))+\mathbb{E}_{y\sim\nu} g(y).
\end{equation}
It uses a direct MLP map and DenseICNN target potential, and is the unregularized
max-correlation ablation of the OTM-style objective
\citep{rout2021generative,tarasov2025statistical}. The results are shown in Figure~\ref{fig:errors_mc}.

\begin{figure}[h]
    \centering

    \begin{subfigure}{0.48\linewidth}
        \centering
        \includegraphics[width=\linewidth]{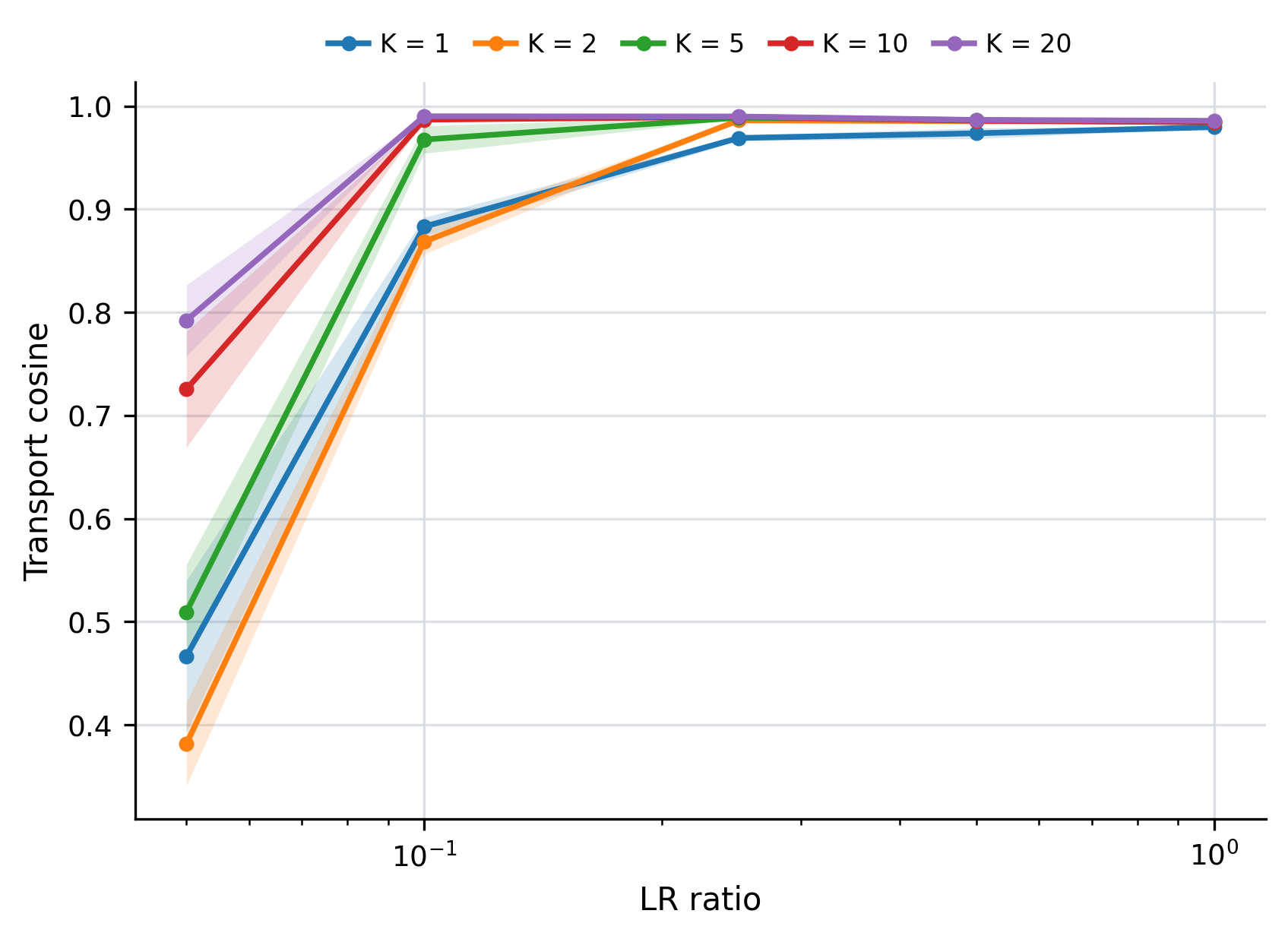}
        \caption{Map cosine similarity.}
        \label{fig:map_error}
    \end{subfigure}
    \hfill
    \begin{subfigure}{0.48\linewidth}
        \centering
        \includegraphics[width=\linewidth]{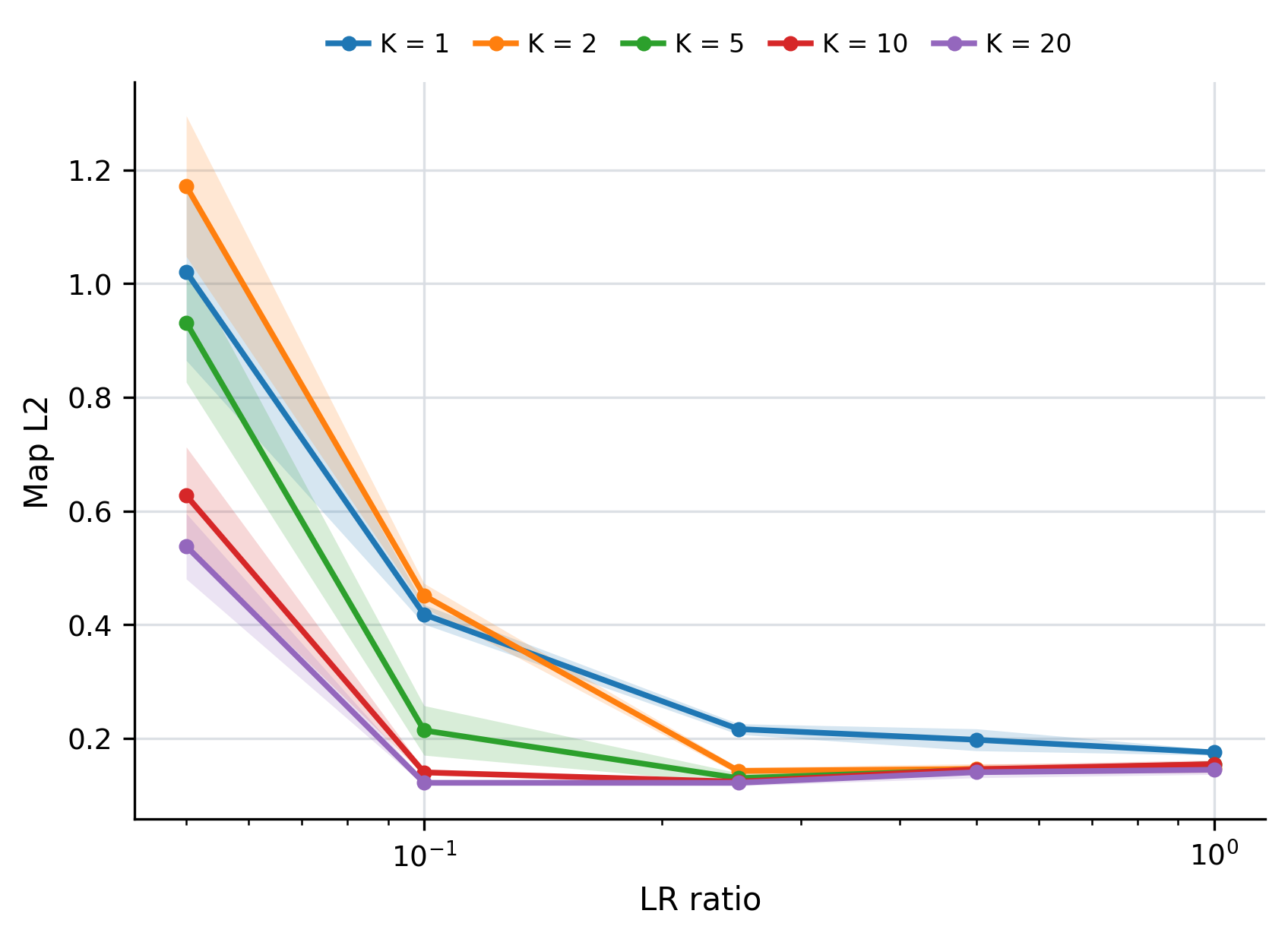}
        \caption{Map error.}
        \label{fig:grad_error}
    \end{subfigure}\\

    \begin{subfigure}{0.48\linewidth}
        \centering
        \includegraphics[width=\linewidth]{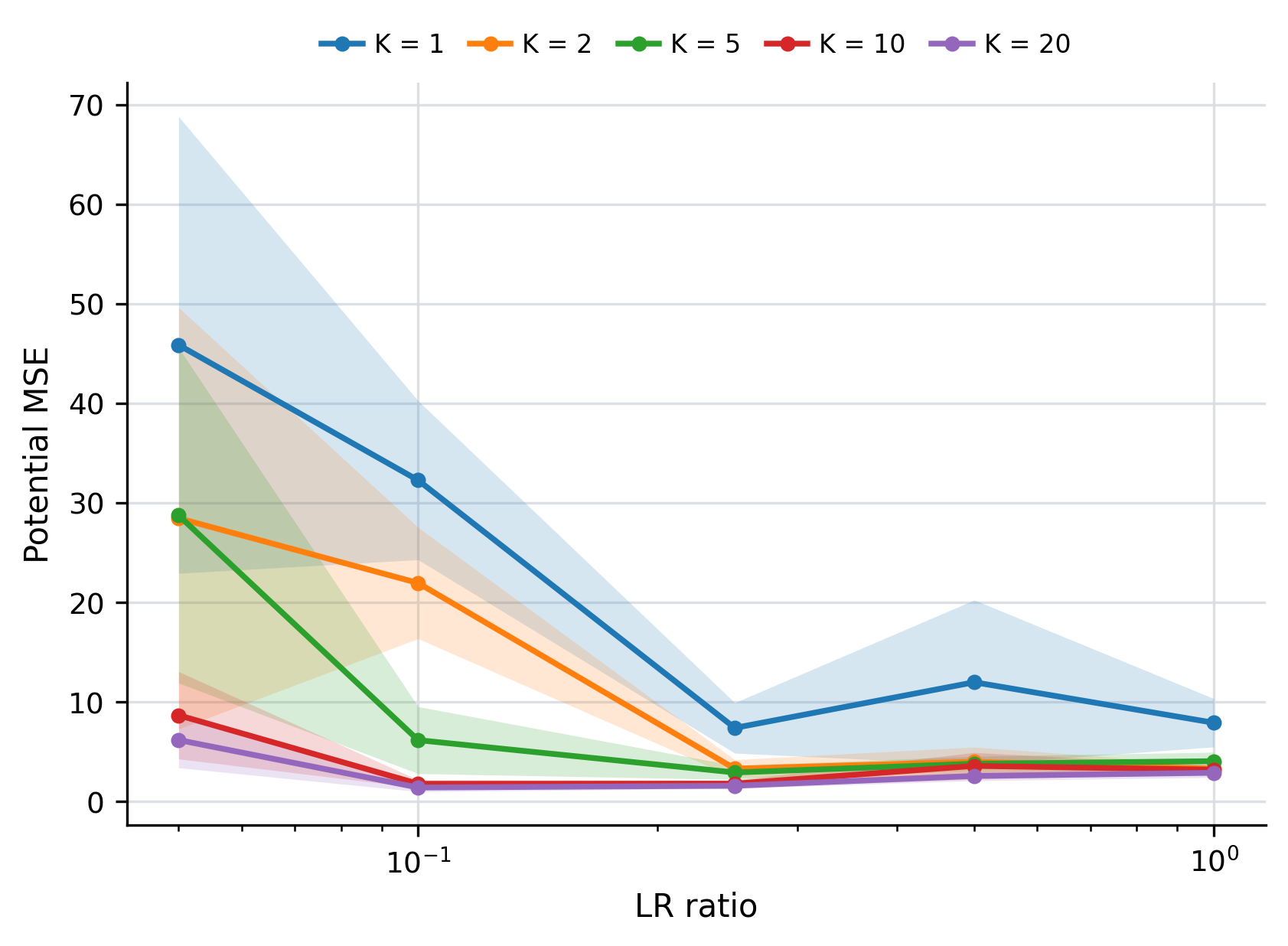}
        \caption{Potential error.}
        \label{fig:map_error}
    \end{subfigure}
    \hfill
    \begin{subfigure}{0.48\linewidth}
        \centering
        \includegraphics[width=\linewidth]{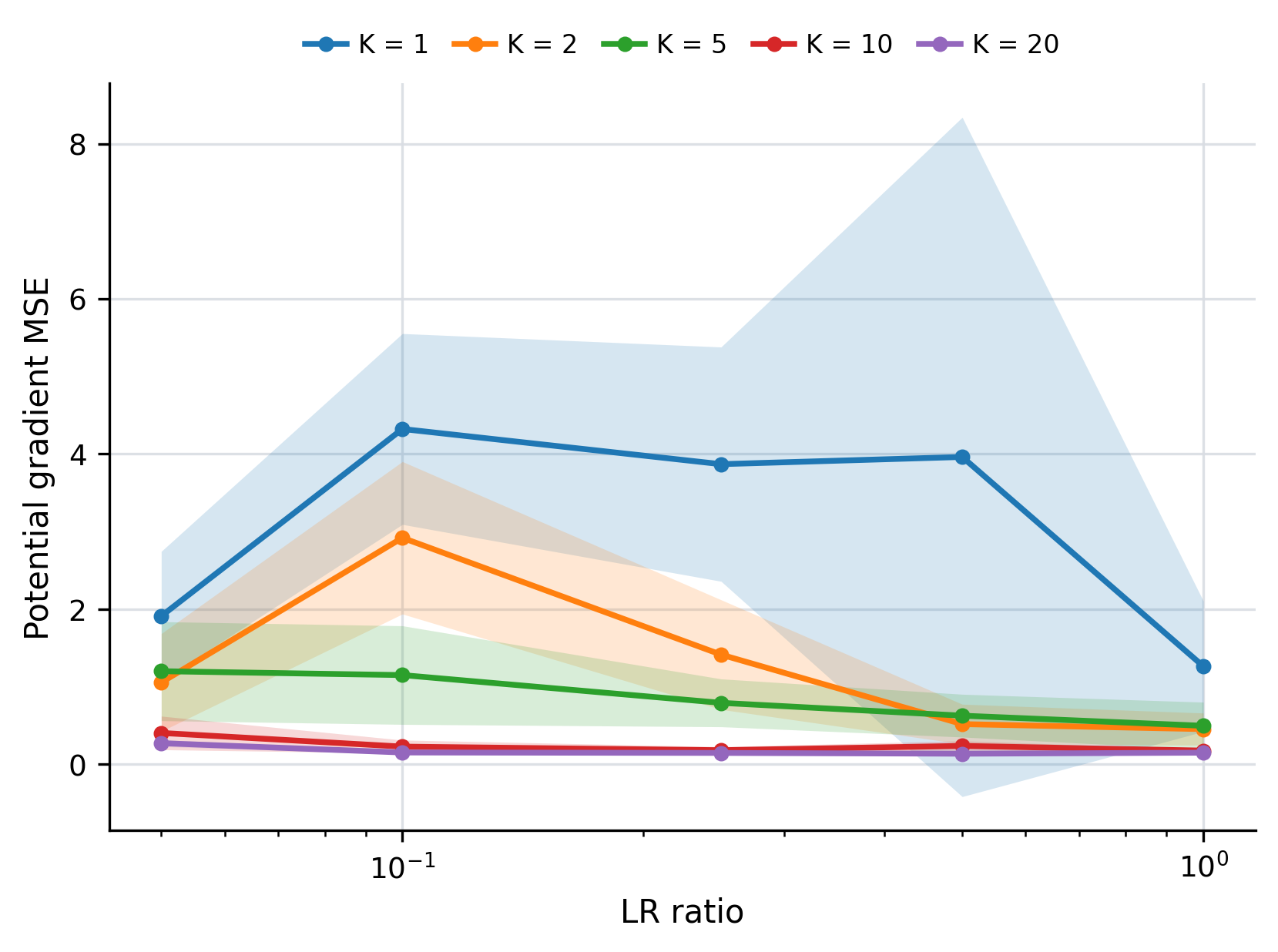}
        \caption{Potential gradient error.}
        \label{fig:grad_error}
    \end{subfigure}\\

    \begin{subfigure}{0.48\linewidth}
        \centering
        \includegraphics[width=\linewidth]{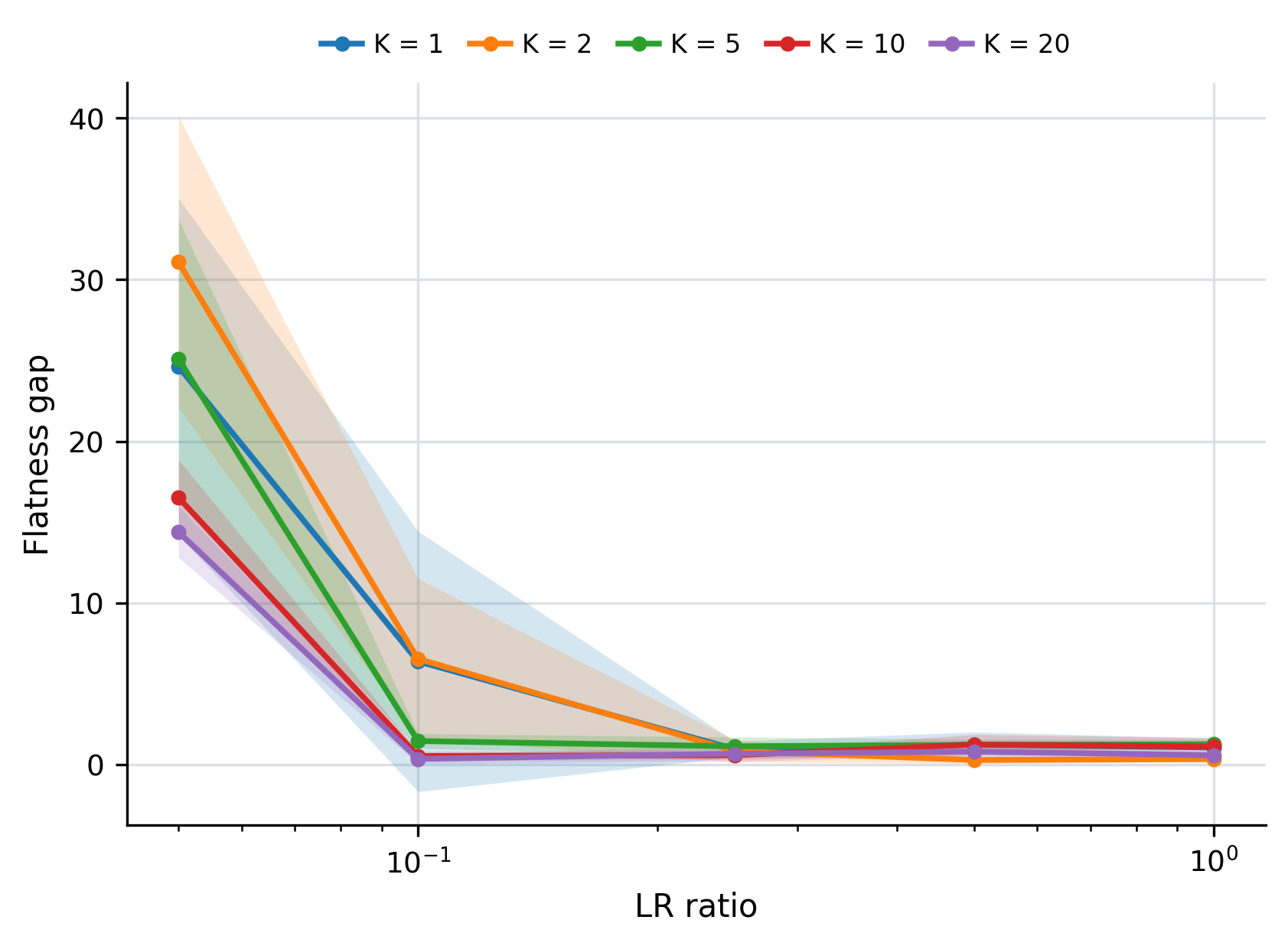}
        \caption{Flatness.}
        \label{fig:map_error}
    \end{subfigure}
    \hfill
    \begin{subfigure}{0.48\linewidth}
        \centering
        \includegraphics[width=\linewidth]{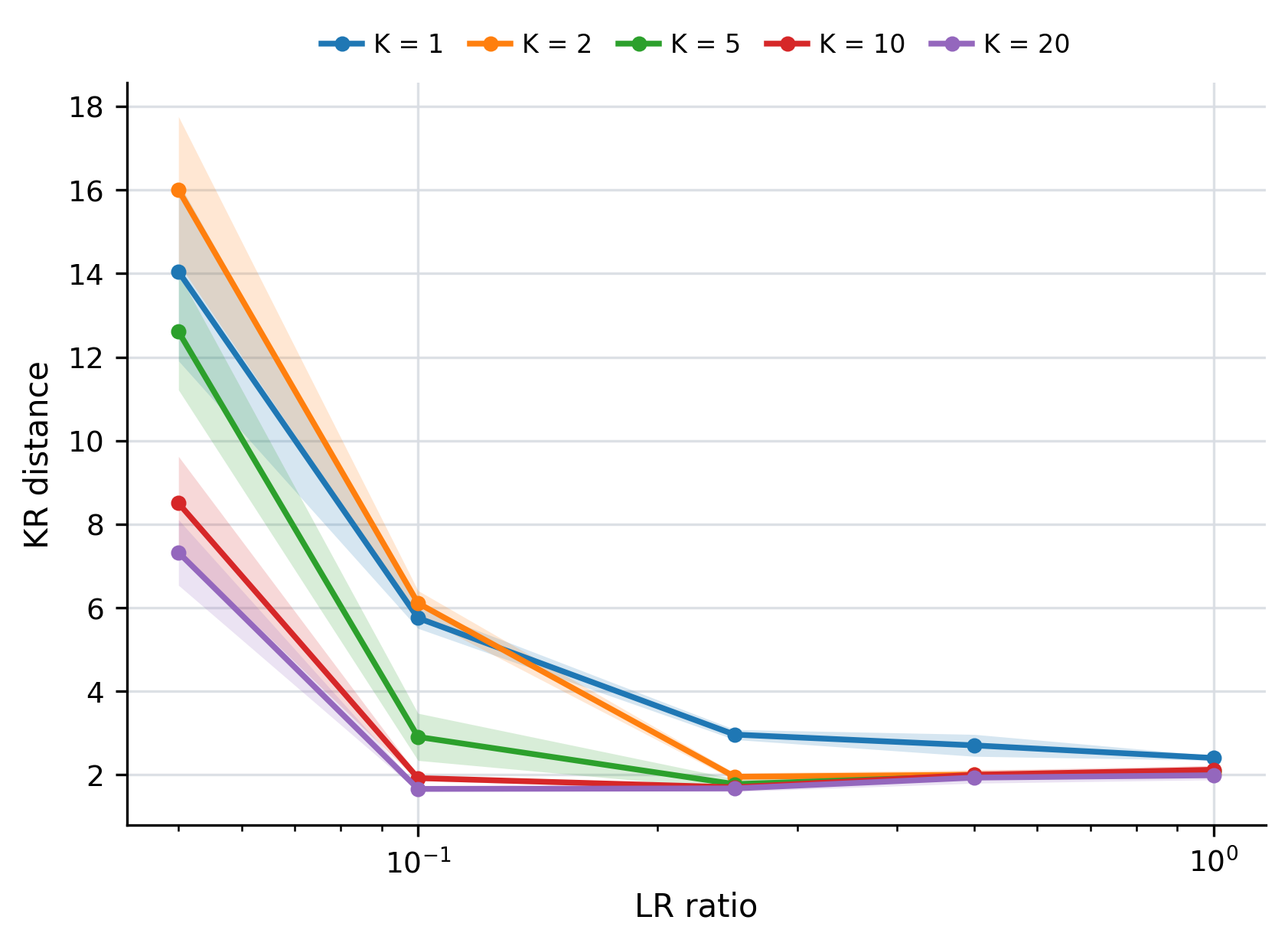}
        \caption{$d_{KR}$.}
        \label{fig:grad_error}
    \end{subfigure}

    \caption{Convergence behavior of the transport map and potential for the Max Correlation method.}
    \label{fig:errors_mc}
\end{figure}

\subsection{OTM}
OTM uses the same max-correlation objective, but adds the published
gradient-optimality penalty from optimal transport modeling
\citep{rout2021generative}. In our setup this penalty has weight $0.1$. The results are shown in Figure~\ref{fig:errors_otm}.

\begin{figure}[h]
    \centering

    \begin{subfigure}{0.48\linewidth}
        \centering
        \includegraphics[width=\linewidth]{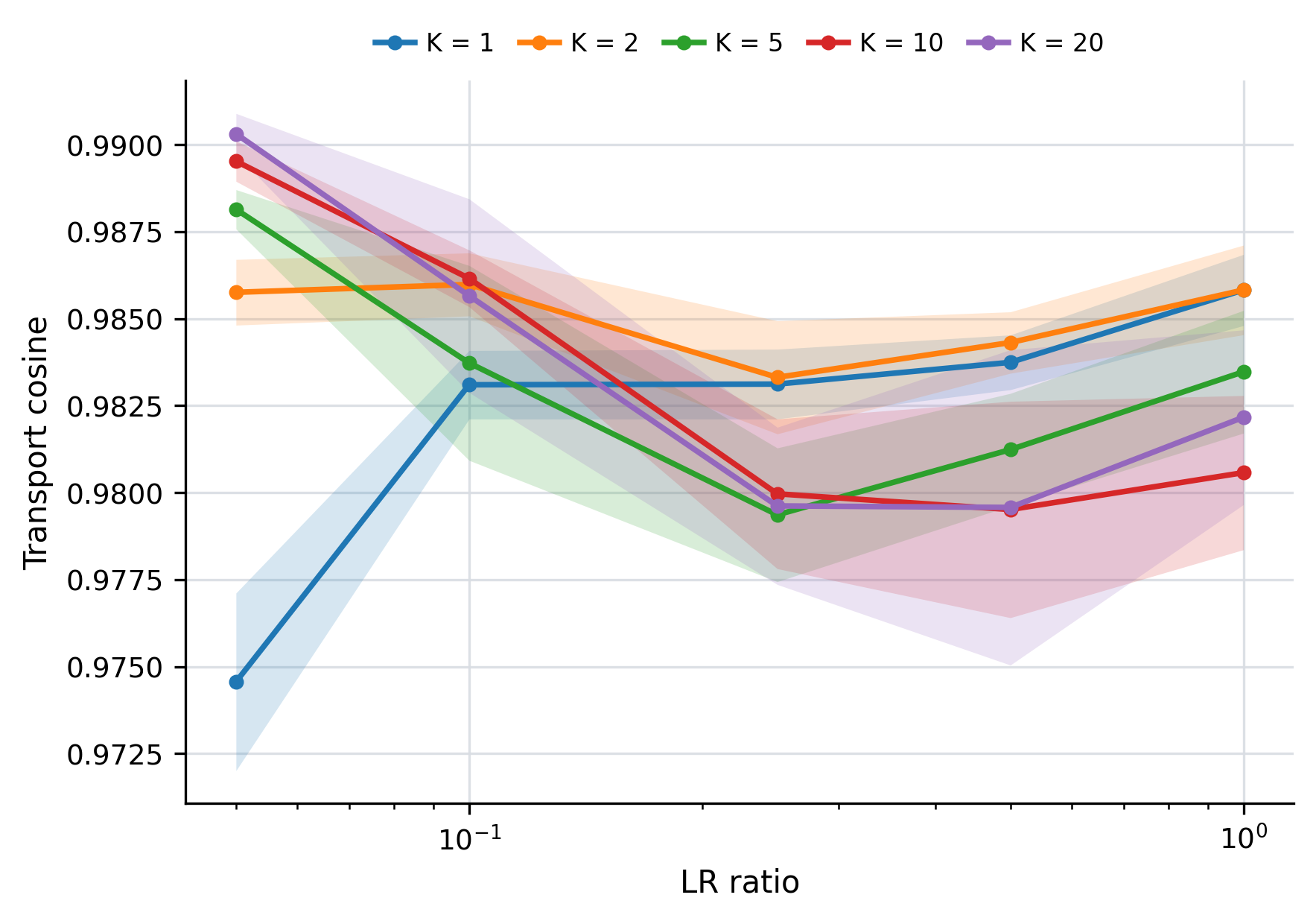}
        \caption{Map cosine similarity.}
        \label{fig:map_error}
    \end{subfigure}
    \hfill
    \begin{subfigure}{0.48\linewidth}
        \centering
        \includegraphics[width=\linewidth]{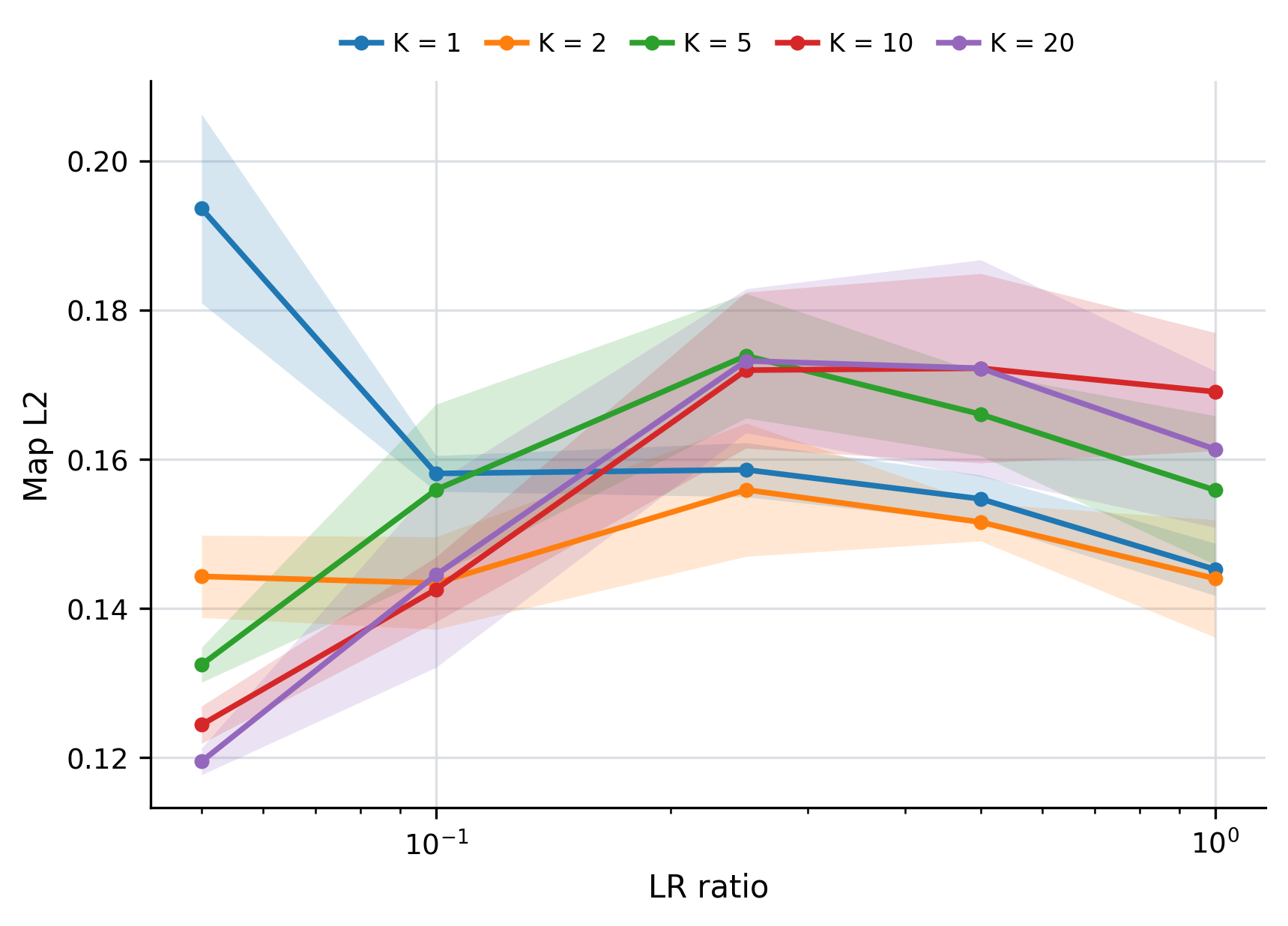}
        \caption{Map error.}
        \label{fig:grad_error}
    \end{subfigure}\\

    \begin{subfigure}{0.48\linewidth}
        \centering
        \includegraphics[width=\linewidth]{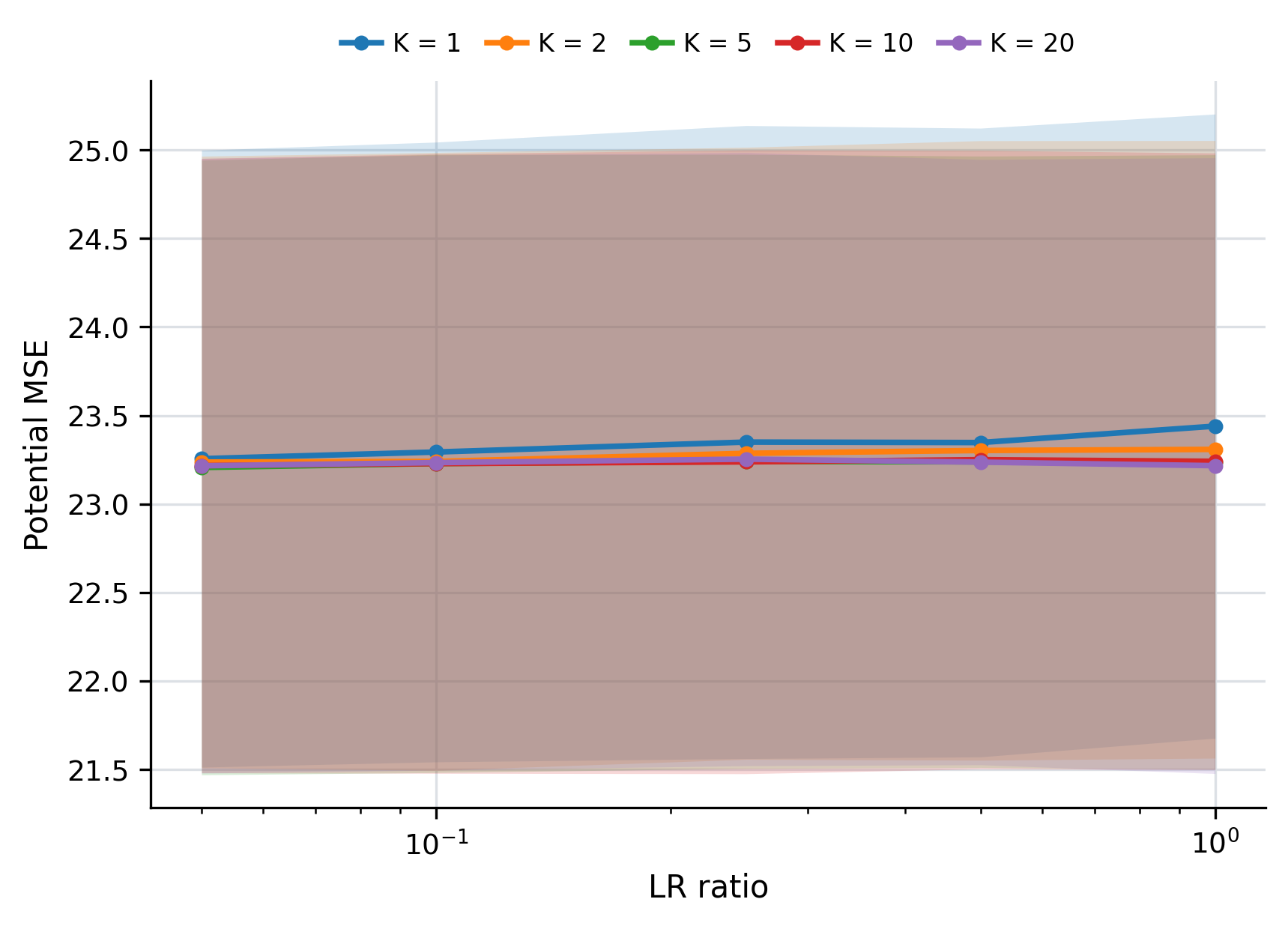}
        \caption{Potential error.}
        \label{fig:map_error}
    \end{subfigure}
    \hfill
    \begin{subfigure}{0.48\linewidth}
        \centering
        \includegraphics[width=\linewidth]{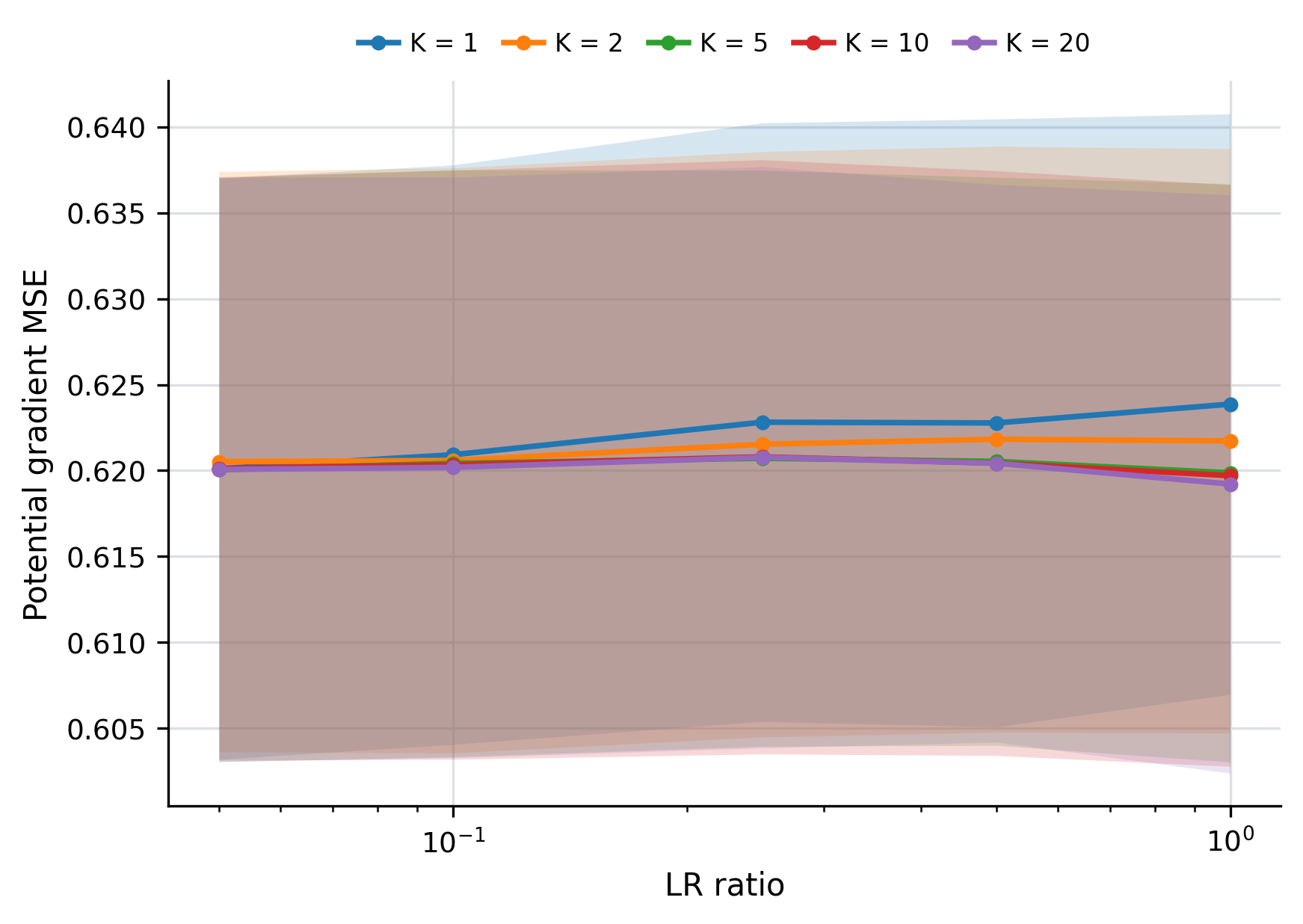}
        \caption{Potential gradient error.}
        \label{fig:grad_error}
    \end{subfigure}\\

    \begin{subfigure}{0.48\linewidth}
        \centering
        \includegraphics[width=\linewidth]{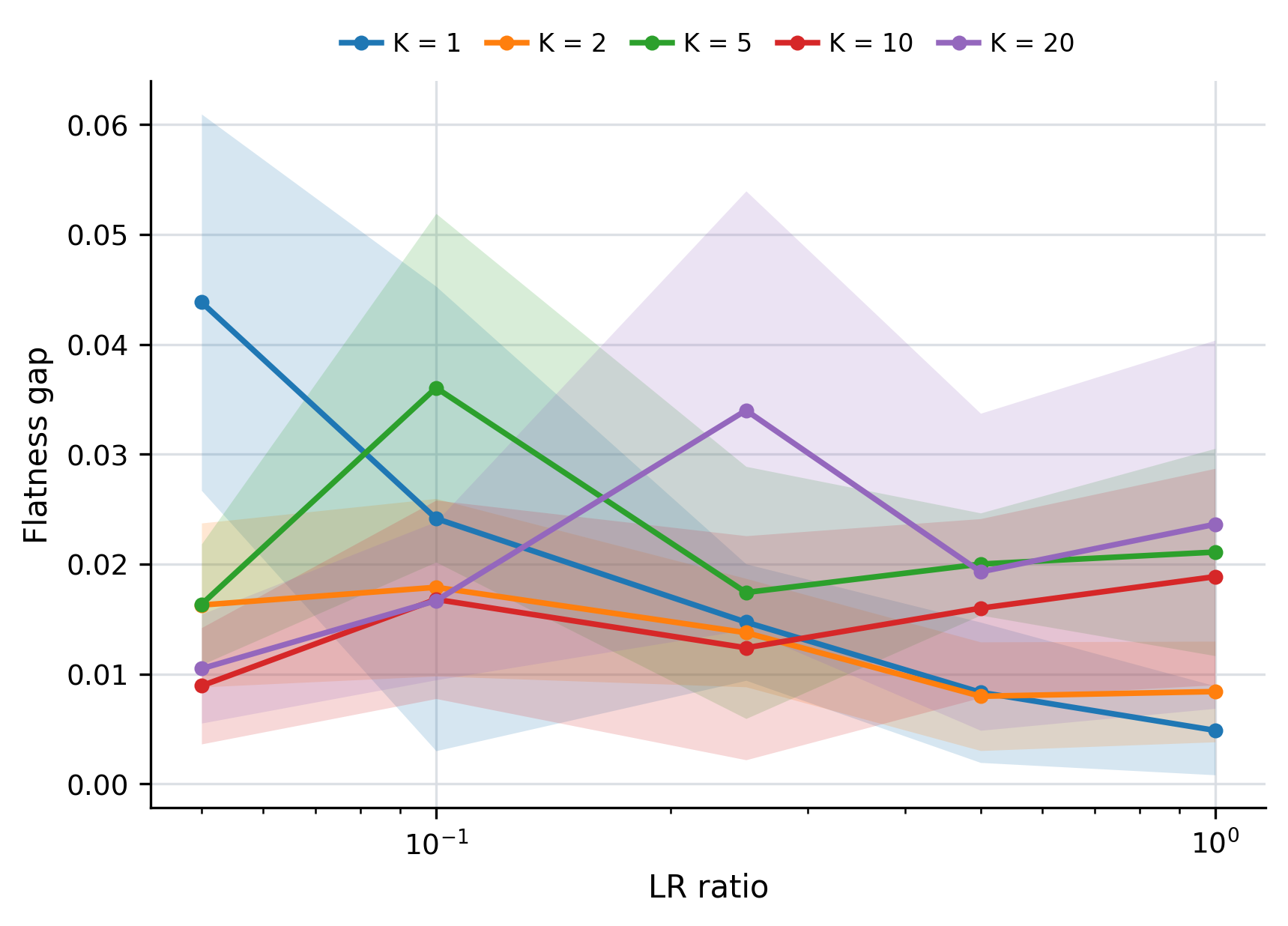}
        \caption{Flatness.}
        \label{fig:map_error}
    \end{subfigure}
    \hfill
    \begin{subfigure}{0.48\linewidth}
        \centering
        \includegraphics[width=\linewidth]{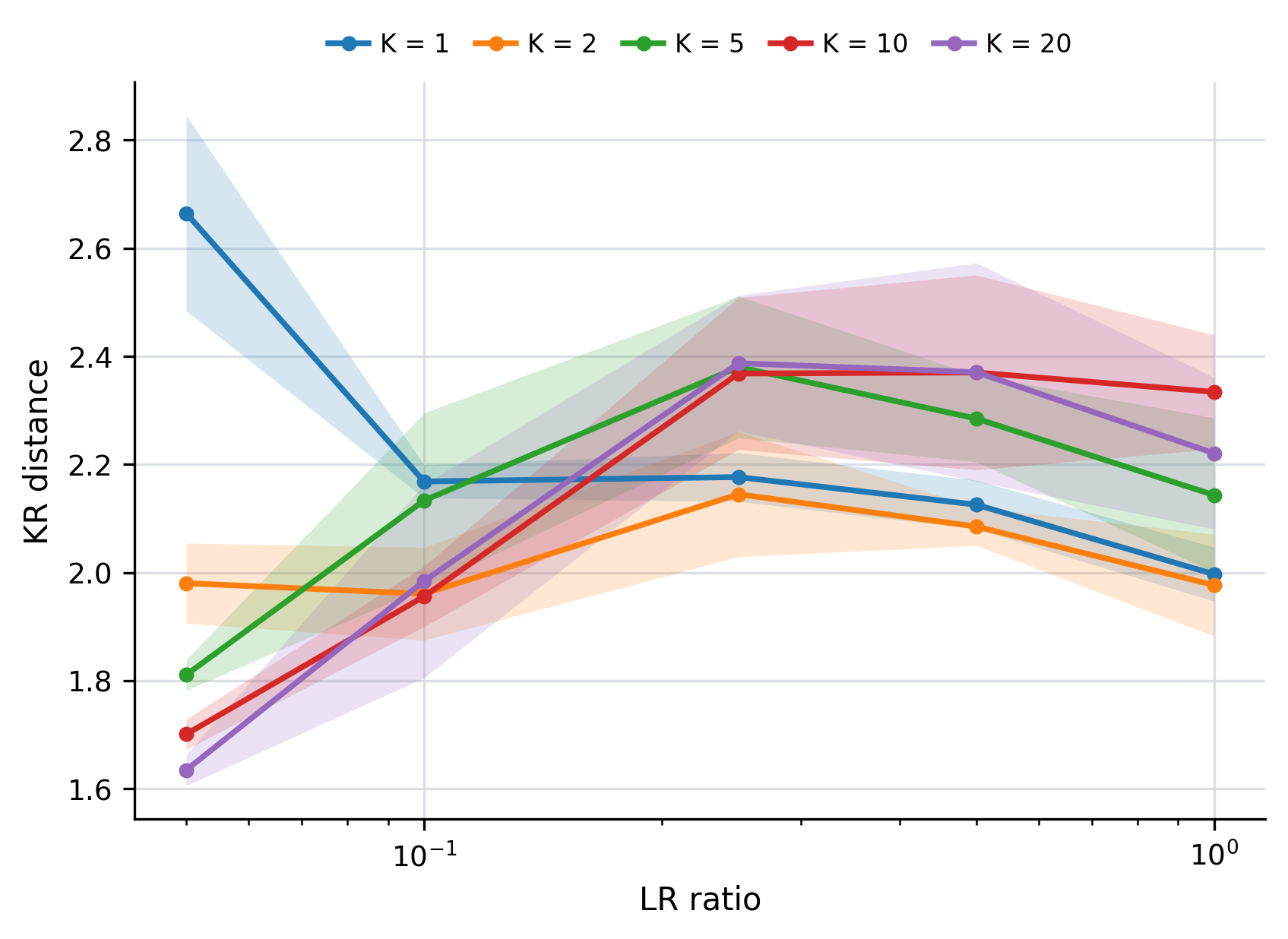}
        \caption{$d_{KR}$.}
        \label{fig:grad_error}
    \end{subfigure}

    \caption{Convergence behavior of the transport map and potential for the OTM method.}
    \label{fig:errors_otm}
\end{figure}

\end{document}